\documentclass[12pt]{amsart} 
\usepackage{amsmath,amssymb,amsfonts,amsthm,amscd,amstext,amsxtra,amsopn,array,url,mathrsfs,enumerate,anysize,enumitem,mathabx} 
\marginsize{2.20cm}{2.20cm}{2.20cm}{2.20cm} 
\usepackage{times} 
\usepackage{amsmath} 
\usepackage{amssymb} 
\usepackage{amsfonts} 
\usepackage{xcolor} 
\usepackage{hyperref} 
\usepackage{amsthm}
\usepackage{tikz-cd}

\makeatletter
\@namedef{subjclassname@2020}{\textup{2020} Mathematics Subject Classification}
\makeatother

\usepackage{amsrefs} 
\newenvironment{rezabib} 
{\bibdiv\biblist\setupbib} 
{\endbiblist\endbibdiv} 

\def\setupbib{\catcode`@=\active} 
\begingroup\lccode`~=`@ 
\lowercase{\endgroup\def~}#1#{\gatherkey{#1}} 
\def\gatherkey#1#2{\gatherkeyaux{#1}#2\gatherkeyaux} 
\def\gatherkeyaux#1#2,#3\gatherkeyaux{\bib{#2}{#1}{#3}}

\newtheorem{theorem}{Theorem}[section] 
\newtheorem{proposition}[theorem]{Proposition}

\newtheorem{corollary}[theorem]{Corollary} 
\newtheorem{problem}[theorem]{Problem} 
\numberwithin{equation}{section} 

\theoremstyle{definition}

\newtheorem{notation}[theorem]{Notations}

\newtheorem{Remark}[theorem]{Remarks} 

\newcommand{\NN}{\mathbb{N}} % Naturals
\newcommand{\QQ}{\mathbb{Q}} % Rationals
\newcommand{\ZZ}{\mathbb{Z}} % Integers

\DeclareMathOperator{\Gal}{Gal}

\DeclareMathOperator{\1}{1}

\DeclareMathOperator{\Aut}{Aut}

\DeclareMathOperator{\modd}{mod}
\DeclareMathOperator{\GL}{GL}

\DeclareMathOperator{\primen}{prime}

\DeclareMathOperator{\li}{li}
\DeclareMathOperator{\Hom}{Hom}
\DeclareMathOperator{\sgn}{sign}

\DeclareMathOperator{\proj}{proj}

\usepackage{mathtools}

\renewcommand{\Re}{{\mathfrak{Re}}}

\begin{document}

\begin{center} 

\title{Constants for Artin-like problems in Kummer and division fields} 

\author{Amir Akbary}
\address{Department of Mathematics and Computer Science, University of Lethbridge, Lethbridge, Alberta T1K 3M4, Canada}
\email{amir.akbary@uleth.ca}

\author{Milad Fakhari} 
\address{Department of Mathematics and Computer Science, University of Lethbridge, Lethbridge, Alberta T1K 3M4, Canada} 
\email{milad.fakhari@uleth.ca}


\date{\today}

\begin{abstract} 
We apply the character sums method of Lenstra, Moree, and Stevenhagen to explicitly compute the constants in 
the Titchmarsh divisor problem for Kummer fields and division fields of Serre curves. We derive our results as special cases of a general result on the product expressions for the sums in the form $$\sum_{n=1}^{\infty}\frac{g(n)}{\#G(n)}$$ in which $g(n)$ is a multiplicative arithmetic function and $\{G(n)\}$ is a certain family of Galois groups. Our results extend the application of the character sums method to the evaluation of constants, such as the Titchmarsh divisor constants, that are not density constants.
\end{abstract} 
\maketitle 
\end{center}

\section{Introduction}\label{sec:intro}

%Throughout this paper, 
Let $a$ be a non-zero integer that is not $\pm1$.  Let $h$ be the largest integer for which $a$ is a perfect $h$-th power. 
In 1927, Emil Artin proposed a conjecture for the density of primes $q$ for which a given integer $a$ is a primitive root modulo $q$. More precisely, Artin conjectured that the density is

\begin{equation}
\label{ArtinConstant}
    A_a=\displaystyle\prod_{p \: \primen}\left(1-\frac{1}{[\mathbb{Q}(\zeta_p, a^{1/p}):\mathbb{Q}]}\right)=\displaystyle\prod_{\substack{p \: \primen \\ p\mid h}}\left(1-\frac{1}{p-1}\right)\displaystyle\prod_{\substack{p \: \primen \\ p\nmid h}}\left(1-\frac{1}{p(p-1)}\right).
\end{equation}
%where $G(p)$ is the Galois group of $\QQ(\zeta_p,a^{1/p})$ over $\QQ$.  
Here $\zeta_p$ is a primitive $p$-th root of unity in a fixed algebraic closure of $\mathbb{Q}$. 
%Note that $G(p)$ depends on $a$, but we suppress the dependence on $a$ in our notation for simplicity. Also, 
Observe that $A_a=0$ if $a$ is a perfect square as $[\mathbb{Q}(\zeta_2, a^{1/2}):\mathbb{Q}]=1$ for such $a$. 

In 1957, computer calculations of the approximate density for various values of $a$ by D. H. Lehmer and E. Lehmer revealed some discrepancies from the conjectured value $A_a$ (see \cite[Section 2]{S}). The reason for these inconsistencies is the dependency between the splitting conditions in \emph{Kummer fields} $\QQ(\zeta_p,a^{1/p})$. 

To deal with these dependencies, Artin suggested an \textit{entanglement correction factor} that appears when $a_{sf}\equiv1~({\rm mod}~4)$, where $a=a_{sf}\cdot b^2$ in which $b$ is the largest integer such that $b^2$ divides $a$
%$a_{sf}$, the square-free part of $a$, is the largest square-free factor of $a$ if $a>0$, %and it is negative of the largest square-free factor of $a$ if $a<0$ 
(see the preface to Artin's collected works \cite{artin}). 
More precisely, 
%{for a non-square integer $a$,}
the corrected conjectured density $\delta_a$ is
\begin{equation}
\label{artin-corrected}
 \delta_a=
    \begin{cases}
    A_a &\quad\quad\text{if }a_{sf}\nequiv1~({\rm mod}~{4}),\\
    E_a\cdot A_a&\quad\quad\text{if }a_{sf}\equiv1~({\rm mod}~{4}),
    \end{cases}
\end{equation}
where 
\begin{equation}
\label{Hooley}
E_a=1-\mu(|a_{sf}|)\displaystyle\prod_{\substack{p \mid h\\ p\mid a_{sf}} }\frac{1}{p-2}\displaystyle\prod_{\substack{p \nmid h\\ p\mid a_{sf}} }\frac{1}{p^2-p-1}.
\end{equation}
Here, $\mu(.)$ is the M\"obius function.
Hooley \cite{Hooley} proved the modified conjecture in 1967 under the assumption of the \emph{Generalized Riemann Hypothesis} (GRH) for the {Kummer fields} $K_n=\mathbb{Q}(\zeta_n, a^{1/n})$ for square-free values of $n$. For any $n$, let $G(n)$ be the Galois group of $K_n/\QQ$. More precisely, Hooley proved, under the GRH, that the primitive root density is
\begin{equation}
\label{hooley-sum}
    \sum_{n=1}^{\infty}\frac{\mu(n)}{\#G(n)},
\end{equation}
and then showed that the above sum equals the corrected conjectured density $\delta_a$ in \eqref{artin-corrected}.

In \cite{lenstra-stevenhagen-moree}, Lenstra, Moree, and Stevenhagen introduced 
a method which allows one to find
%their character sums method in finding 
product expressions for densities in Artin-like problems. Their method directly studies the primes that do not split completely in a Kummer family attached to $a$, without considering the summation expressions such as \eqref{hooley-sum} for the densities. In  \cite[Theorem 4.2]{lenstra-stevenhagen-moree}, they express the correction factor \eqref{Hooley}, when $a$ is non-square and the discriminant $d$ of $K_2=\QQ(a^{1/2})$ is odd (equivalently $a$ is non-square and $a_{sf} \equiv 1$ (mod $4$)), as
\begin{equation}
\label{artin-lenstra-product}
E_a=1+\displaystyle\prod_{p \mid 2d }\frac{-1}{\#G(p)-1}.
%\displaystyle\prod_{p \: \primen}\left(1-\frac{1}{\#G(p)}\right),
\end{equation}
The authors of \cite{lenstra-stevenhagen-moree} achieve this by constructing a quadratic character $\chi=\prod_{p} \chi_p$ of a certain profinite group $A=\prod_{p} A_p$ such that $\ker \chi=\Gal(K_\infty/\mathbb{Q})$, where $K_{\infty}=\bigcup_{n\geq1}K_n$ (see Section 2 for details). They derive \eqref{artin-corrected} as a special case of the following general theorem (\cite[Theorem 3.3]{lenstra-stevenhagen-moree}) in the context of profinite groups.

\begin{theorem}[{Lenstra-Moree-Stevenhagen}]
\label{LMS}
Let $A=\prod_{p} A_p$, with Haar measure $\nu=\prod_{p} \nu_p$, and the quadratic character $\chi=\prod_{p} \chi_p:A\rightarrow \{\pm 1\}$ a non-trivial character obtained from a family of continuous quadratic characters $\chi_p: A_p \rightarrow \{\pm 1\}$, with $\chi_p$ trivial for almost all primes $p$. Then for $G=\ker \chi$ and $S=\prod_p S_p$, a product of $\nu_p$-measurable subsets $S_p\subset A_p$ with $\nu_p(S_p)>0$, we have
$$\frac{\nu(G\bigcap S)}{\nu(G)} =\left(1+\prod_p \frac{1}{\nu_p(S_p)} \int_{S_p} \chi_p d\nu_p  \right)\cdot \frac{\nu(S)}{\nu(A)}.$$
\end{theorem}

The above theorem shows that if $\frac{\nu(G\bigcap S)}{\nu(G)}\neq \frac{\nu(S)}{\nu(A)}$, then the density of $G \bigcap S$ in $G$ can be expressed as the
density of $S$ in $A$ multiplied by a correction factor that
% of $S$ in $A$ can be corrected to give the density of the elements of $S$ in $G$. Moreover, the correction factor 
can be written explicitly in terms of the average of local characters $\chi_p$ over $S_p$. Moreover, since $\nu=\prod_{p} \nu_p$, $S=\prod_p S_p$, and $A=\prod_{p} A_p$, the quotient $\nu(S)/\nu(A)$ can be written as a product over primes $p$.

Our goals in this paper are two-fold. In one direction, in Theorem \ref{main1} and Corollary \ref{cor-after-main1},  we will show how the character sums method of \cite{lenstra-stevenhagen-moree} can be adapted to directly deal with the sum %similar to \eqref{hooley-sum} 
obtained by replacing $\mu(n)$ in \eqref{hooley-sum} by a general multiplicative function $g(n)$. This is an approach different from the one given in Theorem \ref{LMS} in which a density given as a product, i.e., $\nu(S)/\nu(A)$, is corrected to another density, i.e.,  $\nu(G\bigcap S)/\nu(G)$, which is not explicitly given as an infinite sum. In another direction, we describe how the method of \cite{lenstra-stevenhagen-moree}  can be adapted to derive product expressions for the general sums similar to \eqref{hooley-sum} in which $\mu(n)$ is replaced by a multiplicative arithmetic function that could be supported on non-square free integers (all the examples given in \cite{lenstra-stevenhagen-moree} are dealing with arithmetical functions supported on square-free integers). Such arithmetic sums appear naturally on many Artin-like problems (i.e., problems related to distributions of functions of the residual indices of integers modulo primes or subsets of primes). In addition, some of them, such as Titchmarsh divisor problems for families of number fields, are not problems related to the natural density of subsets of integers. In this direction, our Theorem \ref{product-kummer-family} provides a product formula for the constant appearing in the Generalized Artin Problem for multiplicative functions $f$ (see Problem \ref{generalized}) in full generality.

%
%One of our goals in this paper is to describe how the character sums method of \cite{lenstra-stevenhagen-moree} can be adapted to derive product expressions for the general sums similar to \eqref{hooley-sum} in which $\mu(n)$ is replaced by a multiplicative arithmetic function that could be supported on non-square free integers. Such arithmetic sums appear naturally on many Artin-like problems. Some of them, such as Titchmarsh divisor problems for families of number fields, are not problems related to the natural density of subsets of integers.

We continue with our general setup. Let $a=\pm a_0^e$, where 
%$a_0$ is an integer, 
$e$ is the largest positive integer such that $\lvert a\rvert$ is a perfect $e$-th power,
%=\lvert a_0\rvert^e$, 
and $\sgn(a_0)=\sgn(a)$. (Note that the exponent $e$ can differ from the exponent $h$ defined at the beginning of the introduction. For example, if $a=-3^2$, then $e=2$, however, $h=1$.) In our arguments, the integer $a$ is fixed; so we suppress the dependency on $a$ in most of our notations. We fix a solution of the equation $x^2-a_0=0$ and denote it by $a_0^{1/2}$. The quadratic field $K=\QQ(a_0^{1/2})$, the so-called \emph{entanglement field} in \cite[p. 495]{lenstra-stevenhagen-moree}, plays an important role in our arguments (see Section 2 for the justification of the terminology). We denote the discriminant of $K$ by $D$.
Observe that 
for an integer $a \neq 0, \pm 1$, we have three different cases based on the parity of the exponent $e$ and the sign of $a$: 
\begin{itemize}
\item[] (i) \emph{Odd exponent case}, in which $e$ is odd; 
\item[] (ii) \emph{Square case}, in which $e$ is even and $a>0$; 
\item[] (iii) \emph{Twisted case}, in which $e$ is even and $a<0$. 
\end{itemize}

We refer to cases (i) and (ii) as \emph{untwisted} cases. Note that for odd exponent case $K=K_2$, for square case $K_2=\mathbb{Q}$ and $K\neq K_2$, and for twisted case $K_2=\mathbb{Q}(i)$ and $K\neq K_2$. 

For a Kummer family $\{K_n\}$, the Galois elements in $G(n)=\Gal(K_n/\QQ)$ are determined by their actions on the $n$-th roots of $a$ and the $n$-th roots of unity. Thus, any Galois automorphism can be realized as a group automorphism of the multiplicative group $$R_n=\{\alpha\in\overline{\mathbb{Q}}^{\times};\;\alpha^n\in\langle a\rangle\},$$ the group of \emph{$n$-radicals} of $a$. This yields the injective homomorphisms
\begin{equation}
\label{embed}
    r_n:G(n)\to A(n):=\Aut_{\QQ^{\times}\cap R_n}(R_n),
\end{equation}
where $A(n)$ is the group of automorphisms of $R_n$ fixing elements of $\QQ^\times$. For $n=\prod_{p^k\Vert n} p^k$ we have $A(n)\cong \prod_{p^k\Vert n} A(p^k)$. Let $\nu_p(e)$ denote the multiplicity of $p$ in $e$. Let $\Phi(n)$ be the Euler totient function.  
%and $s$ be as defined in Theorem \ref{product-kummer-family}. 
For odd $p$, $$\#A(p^k)= p^{k-\min\{k, \nu_p(e)\}}\Phi(p^k),$$ and for $p=2$,
$$\#A(2^k)=\begin{cases}
2^{k-\min\{k, s-1\}}\Phi(2^k)&\text{if}~e~\text{is~odd~or}~a>0,\\
2^{k-\min\{k, s-1\}}\Phi(2^{k+1})&\text{if}~e~\text{is~even~and}~a<0, 
\end{cases}$$ 
where 
\begin{equation}
\label{s-def}
s=\begin{cases}
\nu_2(e)+1&\text{if}~e~\text{is~odd~or}~ a>0,\\
\nu_2(e)+2& \text{if}~e~\text{is}~\text{even~and}~a<0.
\end{cases}
\end{equation}
In particular, $\#A(p)=\#G(p)$. (See Proposition \ref{new-prop} for a proof of these claims.)

%We also note that $A(p)\cong G(p)$ for prime $p$.

The following theorem, related to the family of Kummer fields $K_n$, gives us the product expressions of a large family of summations involving the orders of the Galois groups of $K_n/\QQ$.
\begin{theorem}
\label{product-kummer-family}
Let $a=\pm a_0^e$, where $a_0$ and $e$ are defined as above, $K=\QQ(a_0^{1/2})$, and let $D$ be the discriminant of $K$.
%$s$ be the smallest integer such that $K\subset K_{2^s}$. 
Let $g$ be a multiplicative arithmetic function such that 
\begin{equation*}
    \sum_{n\geq1}^{\infty}\frac{\lvert g(n)\rvert}{\#G(n)}<\infty,
\end{equation*}
where $G(n) =
{\rm Gal}(\mathbb{Q}(\zeta_n, a^{1/n})/\mathbb{Q})$ for all $n \geq 1$.
%$\{G(n)\}$ is the family of Galois groups of the Kummer family $\{\mathbb{Q}(\zeta_n, a^{1/n}) \}$.
Let 
%$A$ and 
$A(n)$ be as defined above. 
%Let 
%$$s=\begin{cases}
%\nu_2(e)+1&\text{if}~e~\text{is~odd~or}~ a>0,\\
%\nu_2(e)+2& \text{if}~e~is~\text{even~and}~a<0.
%\end{cases}$$
Then, 
\begin{equation}
\label{product-kummer}
    \sum_{n=1}^{\infty}\frac{g(n)}{\#G(n)}=\prod_p\sum_{k\geq 0}\frac{g(p^k)}{\#A(p^k)}+
    \prod_{p}\sum_{k\geq\ell(p)}\frac{g(p^k)}{\#A(p^k)},
 %\prod_{p\nmid 2D}\sum_{k\geq 0}\frac{g(p^k)}{\#A(p^k)},
\end{equation}
where
$$\ell(p)=
\left\{\begin{array}{ll}
0&\text{if~} p~ \text{is~ odd~and~}p\nmid D,\\1&\text{if~} p~ \text{is~ odd~and~}p\mid D,\\
%\left\{\begin{array}{ll}
%1&\text{if~} p~ \text{is~ odd},\\
s&\text{if}~ p=2~ \text{and}~ D ~\text{is odd},\\
\max\{2,s\}&\text{if}~ p=2~ \text{and}~ 4\Vert D,\\

2&\text{if}~ p=2,~8\Vert D,~\text{and}~(\nu_2(e)=1~\text{and}~a<0),\\  
\max\{3,s\}&\text{if}~ p=2,~ 8\Vert D,~\text{and}~(\nu_2(e)\neq 1~\text{or}~a>0).\end{array}
\right.
$$
\end{theorem}

\begin{Remark}
(i) In the summation \eqref{hooley-sum} appearing in Artin's primitive root conjecture, we have $g(n)=\mu(n)$. In this case, formula \eqref{product-kummer} for $g(n)=\mu(n)$ provides a unified way of expressing the constant in Artin's primitive root conjecture as a sum of products over primes.
Note that if $e$ is even and $a>0$, i.e., $a$ is a perfect square, we have $$\sum_{k\geq 0}\frac{\mu(2^k)}{\#A(2^k)}= 0~~{\rm ~~ and}~\sum_{k\geq\ell(2)}\frac{\mu(2^k)}{\#A(2^k)}= 0,$$
where the first sum is zero since $\#A(2)=1$, and the second sum is zero since $\ell(2)\geq 2$. Hence, \eqref{product-kummer} vanishes. Also, if $e$ is even and $a<0$, then $\ell(2)\geq 2$. Thus, \eqref{product-kummer} reduces to \eqref{ArtinConstant}. If $e$ is odd and $D$ is even, then again $\ell(2)\geq 2$ and \eqref{product-kummer} reduces to \eqref{ArtinConstant}. The only remaining case is when $e$ is odd and $D$ is odd (equivalently $e$ odd and $a_{sf}\equiv 1$ (mod $4$)), where \eqref{product-kummer} reduces to $E_a \cdot A_a$ given in \eqref{artin-corrected}.

{(ii) As a consequence of Theorem 1.1, we can derive necessary and sufficient conditions for the vanishing of 
\begin{equation}
\label{constant}
\sum_{n=1}^{\infty}\frac{g(n)}{\#G(n)}.
\end{equation}
More precisely, \eqref{constant} vanishes if and only if one of the following holds:

\noindent (a) For a prime $p\nmid 2D$,  we have $\displaystyle{\sum_{k\geq 0}\frac{g(p^k)}{\#A(p^k)}=0}$.

\noindent (b) We have $$\prod_{p\mid 2D} \sum_{k\geq 0}\frac{g(p^k)}{\#A(p^k)}+ \prod_{p\mid 2D} \sum_{k\geq \ell(p)}\frac{g(p^k)}{\#A(p^k)}=0.$$

In the case of Artin's conjecture, (a) is never satisfied and (b) holds if and only if $a$ is a perfect square.
}

{(iii) If $\#G(n)$ were a multiplicative function, then the sum in \eqref{product-kummer} would have been equal to the {product} $\prod_p\sum_{k\geq 0}\frac{g(p^k)}{\#G(p^k)}$. However, this is not the case for the Kummer family, and thus, the sum in \eqref{product-kummer} may differ from the above naive product. If the sum and the product are not equal, then  
a complex number $E_{a, g}$ is called a \emph{correction factor} if 
\begin{equation*}
    \sum_{n=1}^{\infty}\frac{g(n)}{\#G(n)}=E_{a, g} \prod_p\sum_{k\geq 0}\frac{g(p^k)}{\#G(p^k)}.
\end{equation*}
The expression \eqref{product-kummer} provides precise information on the correction factor $E_{a, g}$. In fact, 
if  \\$\sum_{k\geq0}\frac{g(p^k)}{\#G(p^k)}\neq0$ for all primes $p\mid 2D$, we have
\begin{equation*}
%\label{product-kummer}
    \sum_{n=1}^{\infty}\frac{g(n)}{\#G(n)}=\left(\frac{\displaystyle{\prod_{p\mid 2D}\sum_{k\geq 0}\frac{g(p^k)}{\#A(p^k)}+    \prod_{p\mid 2D}\sum_{k\geq\ell(p)}\frac{g(p^k)}{\#A(p^k)}}}{\displaystyle{\prod_{p\mid 2D}\sum_{k\geq 0} \frac{g(p^k)}{\#G(p^k)}}}\right)\prod_p \sum_{k\geq 0}\frac{g(p^k)}{\#G(p^k)}.
\end{equation*}

On the other hand, if $\sum_{k\geq0}\frac{g(p^k)}{\#G(p^k)}=0$ for some prime $p\mid 2D$, and $\sum_{n\geq1}\frac{g(n)}{\#G(n)}\neq 0$, then 
the {product} $\prod_p\sum_{k\geq 0}\frac{g(p^k)}{\#G(p^k)}$ cannot be corrected.

%It should be noted that for $K=\mathbb{Q}(\sqrt{\pm 2})$ the above correction factor is slightly different from the one given in Theorem \ref{LMS} for the density problems since in these cases $\#G(2^k)\neq \#A(2^k)$ for some positive integers $k$.
}

(iv) For integer $a \neq 0, \pm 1$, let $n_a=\prod_{p\mid 2D} p^{\ell(p)}$, where $D$ and $\ell(p)$ are as in Theorem \ref{product-kummer-family}. Then, by taking $g(n)=1/n^z$, for $\Re(z)>0$, in Theorem \ref{product-kummer-family} and comparing the coefficients of $1/n^z$ in both sides of \eqref{product-kummer}, we get
$$
[\mathbb{Q}(\zeta_n, a^{1/n}):\mathbb{Q}]=\begin{cases}
\#A(n)&\text{if } n_a\nmid n,\\
\frac{1}{2}\#A(n)&\text{if } n_a\mid n.
\end{cases}
$$

\end{Remark}

\medskip\par

The formula \eqref{product-kummer} can be used to study the constants in many Artin-like problems.  We next apply this formula in the computation of the average value of a specific arithmetic function attached to a Kummer family. 
More precisely, for $\{K_n:=\QQ(\zeta_n,a^{1/n})\}_{n\geq 1}$, we define
\begin{equation*}
    \tau_{a}(p)=\#\left\{n\in\NN ;\; p\text{ splits completely in }K_n/\QQ\right\}.   
\end{equation*}
The Titchmarsh divisor problem attached to a Kummer family concerns the behaviour of $\sum_{p\leq x}\tau_a(p)$ as $x\to\infty$ (see \cite{AG} for the motivation behind this problem and its relation with the classical Titchmarsh divisor problem on the average value of the number of divisors of shifted primes). Under the assumption of the GRH for the Dedekind zeta function of $K_n/\QQ$ for $n\geq1$, Felix and Murty \cite[Theorem 1.6]{felix-murty} proved that  
\begin{equation}
\label{felix-murty-tdp}
\sum_{p\leq x}\tau_a(p)\sim \left(\sum_{n\geq1}\frac{1}{[K_n:\QQ]}\right)\cdot \li(x),
\end{equation}
as $x\to\infty$, where $\li(x)=\int_2^x \frac{1}{\log t}dt$. They do not provide an Euler product expression for the constant appearing in the main term of \eqref{felix-murty-tdp}. As a direct consequence of 
Theorem \ref{product-kummer-family} with $g(n)=1$, we readily find an explicit product formula for the constant appearing in \eqref{felix-murty-tdp}. 
%the Titchmarsh Divisor Problem attached to a Kummer family. 
\begin{proposition}
\label{TDPK-formula}
Let $a=\pm a_{0}^e$ with  $e=\prod_pp^{\nu_p(e)}$, and let $D$ be the discriminant of $K=\QQ(a_0^{1/2})$. Then, if $e$ is odd or $a>0$, 
\begin{equation}
    \label{kummer-tdp-lastproduct}
    \begin{split}
    \sum_{n\geq1}\frac{1}{[K_n:\QQ]}=& \left(1+\frac{c_0}{3\cdot2^{\nu_2(e)}-2}\prod_{p\mid2D}\frac{p^{\nu_p(e)+2}+p^{\nu_p(e)+1}-p^2}{p^{\nu_p(e)+3}+p^{\nu_p(e)}-p^2}\right)\\
    &~~~~~~~~~~~\times\prod_p\left(1+\frac{p^{\nu_p(e)+2}+p^{\nu_p(e)+1}-p^2}{p^{\nu_p(e)}(p-1)(p^2-1)}\right),
    \end{split}
\end{equation}
where% if $(a_0)_{sf}\neq\pm2$,
\begin{equation*}
    c_0=
    \begin{cases}
    1/4 & \text{ if } 4\Vert D\;\text{and}\;\nu_2(e)=0,\text{ or if } 8\Vert D ~\text{and}~ \nu_2(e)=1, \\
    1/16 & \text{ if } 8\Vert D\;\text{and}\; \nu_2(e)=0,\\
    1 & \text{ otherwise.}
    \end{cases}
\end{equation*}
% if $e$ is odd or $a>0$ (i.e., the odd exponent case or the square case), and i
If $e$ is even and $a<0$ (i.e., the twisted case)
 \begin{equation}
    \label{kummer-tdp-lastproduct-2}
    \begin{split}
    \sum_{n\geq1}\frac{1}{[K_n:\QQ]}=& \left(1+\frac{c_0}{3\cdot2^{\nu_2(e)+2}-2}\prod_{\substack{p\mid D\\ p\neq 2}}\frac{p^{\nu_p(e)+2}+p^{\nu_p(e)+1}-p^2}{p^{\nu_p(e)+3}+p^{\nu_p(e)}-p^2}\right)\\
    &~~~~~~~~~~~\times    \left(1+\frac{2^{\nu_2(e)+2}-2^{\nu_2(e)}-1}{3\cdot 2^{\nu_2(e)}}\right)\prod_{\substack{p\\p\neq 2}}\left(1+\frac{p^{\nu_p(e)+2}+p^{\nu_p(e)+1}-p^2}{p^{\nu_p(e)}(p-1)(p^2-1)}\right),
    \end{split}
\end{equation}
where 

\begin{equation*}
    c_0=
    \begin{cases}
    4 & \text{ if } 8\Vert D\;\text{and}\; \nu_2(e)=1,\\
    1 & \text{ otherwise.}
    \end{cases}
\end{equation*}

\end{proposition}

Let $c_{a}$ denote the constant given in \eqref{kummer-tdp-lastproduct} and \eqref{kummer-tdp-lastproduct-2}. 
%From \eqref{kummer-tdp-lastproduct} 
We can write $c_a=q_a\cdot u$, where $q_a$ is a rational number depending on $a$, and $u$ is the \emph{universal constant}
\begin{equation}
\label{universal}
\sum_{n=1}^{\infty} \frac{1}{n \Phi(n)}=\prod_{p}\left( 1+\frac{p}{(p-1)(p^2-1)} \right)= 2.203856\cdots,
\end{equation}
where $\Phi(n)$ is the Euler totient function. 
Note that $n\Phi(n)$ is the ``generic/expected" degree
of the extension $\mathbb{Q}(\zeta_n, a^{1/n})$ over $\mathbb{Q}$; however, this may not be the case due to entanglement phenomena. Moreover,
observe that if $\nu_p(e)=0$ for all $p$, the expressions for products over all primes $p$  given in \eqref{kummer-tdp-lastproduct} and \eqref{kummer-tdp-lastproduct-2} 
reduce to \eqref{universal}. This is in accordance with \cite[Theorem 1.4]{AF} in which \eqref{universal} appears as the average constant while varying $a$. Thus, on average over $a$
%, a smooth version of  \eqref{kummer-tdp-lastproduct} and \eqref{kummer-tdp-lastproduct-2}, i.e., 
the universal constant appears.
The product expressions of Proposition \ref{TDPK-formula} provide a convenient way of computing the numerical value of $c_{a}$ for a given value of $a$. For example $c_2=c_{-2}=2.258\cdots$.
%
%We record a sample of such values in the following table.
%\medskip\par
%\begin{tabular}{c|cccccccccccc}
%$a$&-13&-10&-8&-5&-3&-2&2&3&5&8&10&13\\
%\hline
%$c_{a}$
%&2.205&2.206&2.972&2.214&2.343&2.258&2.258&2.238&2.247&2.972&2.206&2.209
%\end{tabular}
%\medskip\par

The classical Artin conjecture and the Titchmarsh divisor problem for a Kummer family are instances of a more general problem that we now describe.   For an integer $a \neq 0, \pm 1$ and a prime $p\nmid a$,  the \emph{residual index} of $a$ mod $p$, denoted by $i_a(p)$ is the index of the subgroup $\langle a \rangle$ in the multiplicative group $(\mathbb{Z}/p\mathbb{Z})^\times$.
There is a vast amount of literature on the study of asymptotics of functions of $i_a(p)$ as $p$ varies over primes. In \cite[p. 377]{Papa}, the following problem is proposed.

\begin{problem}[Generalized Artin Problem]
\label{generalized}
%For certain 
Determine integers $a$ and arithmetic functions $f(n)$ for which the asymptotic formula $$\sum_{p\leq x} f(i_a(p)) \sim c_{f, a} \li(x),$$
as $x\rightarrow \infty$, hold, where 
\begin{equation}
\label{series}
c_{f, a}:= \sum_{n\geq 1} \frac{g(n)}{[K_n:\mathbb{Q}]}.
\end{equation}
Here $g(n)=\sum_{d\mid n} \mu(d) f(n/d)$ is the M\"{o}bius inverse of $f(n)$, where $\mu(n)$ is the M\"{o}bius function.
\end{problem}
Note that by setting $f(n)$ as the characteristic function of the set $S=\{1\}$, hence $g(n)=\mu(n)$, in Problem \ref{generalized}, we get the Artin conjecture, and $f(n)=d(n)$ (the divisor function), hence $g(n)=1$,  gives the Titchmarsh divisor problem for a Kummer family, this is true since $\tau_a(p)=d(i_a(p))$ (see \cite[Lemma 2.1]{felix-murty} for details). Also, a conjecture of Laxton from 1969 (see \cite{L} and 
\cite[p. 313]{stephens}) predicts that for $f(n)=1/n$, the generalized Artin problem determines the density of primes in the sequence given by the recurrence $w_{n+2}=(a+1) w_{n+1}-aw_n$, where $a>1$ is a fixed integer. Another instance of Problem \ref{generalized} appears in a conjecture of Bach, Lukes, Shallit, and Williams \cite{BLSW} in which the constant $c_{f, 2}$ for $f(n)=\log{n}$ appears in the main term of the asymptotic formula for $\log{P_2(x)}$, where $P_2(x)$ is the smallest \emph{$x$-pseudopower} of the base $2$.

A notable result on the Generalized Artin Problem, due to Felix and Murty \cite[Theorem 1.7]{felix-murty}, establishes, under the assumption of GRH, the asymptotic  
\begin{equation}
\label{FM}
\sum_{p\leq x} f(i_a(p))= c_{f, a} \li(x)+O_a\left( \frac{x}{(\log{x})^{2-\epsilon-\alpha}} \right),
\end{equation}
for $\epsilon>0$. Here $f(n)$ is an arithmetic function whose M\"obius inverse $g(n)$ satisfies
$$|g(n)| \ll d_k(n)^r (\log{n})^{\alpha}, $$
with $k, r \in \mathbb{N}$ and $0\leq \alpha <1$ all fixed, where $d_k(n)$ denotes the number of representations of $n$ as product of $k$ positive integers.
Observe that 
the identity \eqref{product-kummer} in Theorem \ref{product-kummer-family} conveniently furnishes a product formula in full generality
for the constant $c_{f, a}$ in \eqref{FM} when $f$ (equivalently $g$) is a multiplicative function. This product formula is valuable for studying the vanishing criteria for $c_{f, a}$ and their numerical evaluations for different $f$.

We now comment on the proof of Theorem \ref{product-kummer-family}.  Observe that corresponding to the Kummer family $\{K_n\}$, we can consider the inverse systems $((G(n))_{n\in \NN}, (i_{n_1,n_2})_{n_1\mid n_2})$ and $((A(n))_{n\in \NN},(j_{n_1,n_2})_{n_1\mid n_2})$ ordered by divisibility relation on $\mathbb{N}$, where $G(n)$ and $A(n)$ are as defined before and $i_{n_1,n_2}:G(n_2)\rightarrow G(n_1)$ and $j_{n_1, n_2}: A(n_2) \rightarrow A(n_1)$, for $n_1\mid n_2$, are restriction maps. By taking the inverse limits on both sides of \eqref{embed}
 %from both sides of the injective homomorphisms $r_n$, 
 we have the injective continuous homomorphism
$$r:G=\varprojlim G(n)\to A=\varprojlim A(n)$$
of profinite groups, where $G=\Gal(K_{\infty}/\QQ)$ and $A=\Aut_{\QQ^{\times}\cap R_{\infty}}(R_{\infty})$ with $K_{\infty}=\bigcup_{n\geq1}K_n$ and $R_{\infty}=\bigcup_{n\geq 1}R_n$. 
As profinite groups, both $G$ and $A$ are endowed with topologies that make $G$ and $A$ into compact topological spaces, and thus, they can be equipped by Haar measures. 
We will show that Theorem \ref{product-kummer-family} is a corollary of the following theorem attached to a general setting of profinite groups $G$ and $A$.
\begin{theorem}
\label{main1}
Let $((G(n))_{n\in \NN}, (i_{n_1,n_2})_{n_1\mid n_2})$ and $((A(n))_{n\in \NN},(j_{n_1,n_2})_{n_1\mid n_2})$ be surjective inverse systems of finite groups ordered by divisibility relation on $\mathbb{N}$.
Moreover, for $n\geq 1$, assume that there are injective maps $r_n:G(n)\to A(n)$ compatible with surjective transition maps $i_{n_1,n_2}$ and $j_{n_1,n_2}$, i.e., for $n_1\mid n_2$, the diagram 
\begin{equation*}
    \begin{tikzcd}
%    G(m)\arrow[r,"i_{n,m}"]\arrow[d,"r_m"] & G(n)\arrow[d,"r_n"]\\
%    A(m)\arrow[r,"j_{n,m}"] & A(n)
    G(n_2)\arrow[r,"r_{n_2}"]\arrow[d,"i_{n_1,n_2}"] & A(n_2)\arrow[d,"j_{n_1,n_2}"]\\
    G(n_1)\arrow[r,"r_{n_1}"] & A(n_1)
    \end{tikzcd}
\end{equation*}
commutes. Let $r: G=\varprojlim G(n)\to A=\varprojlim A(n)$ be the resulting injective continuous homomorphism of profinite groups. 
%Let $\{G(n)\}_{n\in\NN}$ and $\{A(n)\}_{n\in\NN}$ be inverse systems of finite groups. Let $r_n: G(n)\to A(n)$ be injective homomorphisms. 
%Let $G=\varprojlim G(n)$ and $A=\varprojlim A(n)$. 
Let $\mu_m$ be the multiplicative group of  $m$-th roots of unity for a fixed $m$. 
Suppose there exists an exact sequence 
\begin{equation}
\label{ev1}
    1\to G\stackrel{r}{\longrightarrow} A\stackrel{\chi}{\longrightarrow} \mu_m\to 1,
\end{equation}
where $\chi$ is a continuous homomorphism.
% where $r$ is the injective homomorphism obtained 
%by taking inverse limit over both sides of homomorphisms 
%from $r_n$'s. 
Let $g$ be an arithmetic function such that 
\begin{equation*}
    \sum_{n\geq1}\frac{\lvert g(n)\rvert}{\#G(n)}<\infty.
\end{equation*}
Consider the natural projections 
    $\pi_{A,n}:A\to A(n)$
and let 
%assume that 
\begin{equation}
\label{g-tilde}
    \Tilde{g}=\sum_{n\geq1}g(n)1_{\ker\pi_{A,n}},
\end{equation}
where $1_{\ker\pi_{A, n}}$ denotes the characteristic function of $\ker\pi_{A, n}$.
%defines a measurable function on $A$. 
Let $\nu_A$ be the normalized Haar measure attached to $A$.  
Then, $\Tilde{g}\in L^1(\nu_A)$ (the space of $\nu_A$-integrable functions)  and
%there exists a character $\chi:A\to\mu_n$ such that
\begin{equation*}
    \sum_{n\geq1}\frac{g(n)}{\#G(n)}=\sum_{i=0}^{m-1}\int_A \Tilde{g}\chi^id\nu_A.
\end{equation*}
\end{theorem}

%\begin{remark}
%Theorem \ref{main1} is more generally true for the inverse systems of finite groups ordered by any directed partially ordered relation on $\mathbb{N}$.
%\end{remark}

Observe that Theorem \ref{main1} is quite general and can be applied in the evaluation of sums of the form $\sum_{n\geq 1} g(n)/\#G(n)$ for any family $\{G(n)\}$ of finite groups satisfying the assumptions of the theorem. 
%and a multiplicative function $g(n)$ of appropriate size, as long as the inverse limit of the $(G(n))_{n\in \mathbb{N}}$ is a profinite group $G$ for which there exists a profinite group $A$, an embedding $r: G \rightarrow A$, and a character $\chi$  that satisfy the exact sequence \eqref{ev1}. 
The family of Galois groups of a Kummer family is an instance of such families. Another example is the family of Galois groups of the division fields attached to a \emph{Serre elliptic curve} $E$ (see Section \ref{Serre section} for the definition). Following \cite[Section 8]{lenstra-stevenhagen-moree}, in Section \ref{Serre section}, we show that the family of Galois groups of the division fields $\{\mathbb{Q}(E[n])\}$ attached to a Serre curve $E$ satisfies the conditions of Theorem \ref{main1} and we deduce the following proposition.
\begin{proposition}
\label{prop-Serre}
Let $\mathbb{Q}(E[n])$ denote the $n$-division field of a Serre elliptic curve defined over $\mathbb{Q}$.
%by a Weierstrass equation $y^2=x^3+ax+b$. 
Let $\Delta$ be the discriminant of any Weierstrass model for $E$.
%the cubic equation $x^3+ax+b=0$ 
Let $D$ be the discriminant of the quadratic field $K=\mathbb{Q}({\Delta}^{1/2})$. Let $g(n)$ be a multiplicative arithmetic function such that 
\begin{equation*}
    \sum_{n\geq1}^{\infty}\frac{\lvert g(n)\rvert}{[\mathbb{Q}(E[n]) : \mathbb{Q}]}<\infty.
\end{equation*}
%$\sum_{k\geq 1}| g(p^k)|/p^{4k-3} (p^2-1)(p-1)\neq -1$ for any prime $p\mid 2D$. 
Then,
\begin{equation}
\label{serre-sum}
  \sum_{n=1}^{\infty}\frac{g(n)}{[\mathbb{Q}(E[n]) : \mathbb{Q}]}  =\prod_p\sum_{k\geq 0}\frac{g(p^k)}{\#\Aut(E[p^k])}+
    \prod_{p}\sum_{k\geq\ell(p)}\frac{g(p^k)}{\#\Aut(E[p^k])},
 %\prod_{p\nmid 2D}\sum_{k\geq 0}\frac{g(p^k)}{\#A(p^k)},
\end{equation}
where 
%in the product on primes dividing $2D$, 
$$\ell(p)=\left\{\begin{array}{ll}
0&\text{if~} p~ \text{is~ odd~and~}p\nmid D,\\
1&\text{if~} p~ \text{is~ odd~and~}p\mid D,\\
1&\text{if}~ p=2~ \text{and}~ D ~\text{is odd},\\
2&\text{if}~ p=2~ \text{and}~ 4\Vert D,\\
3&\text{if}~ p=2~ \text{and}~ 8\Vert D,\\ 
\end{array}
\right.
$$
and
\begin{equation*}
%\label{Apk}
\#\Aut(E[p^k])=\begin{cases}
%p^{k'}\phi(p^k)=
p^{4k-3}(p^2-1)(p-1)&\text{if }k\geq 1,\\
1&\text{if }k=0.
\end{cases}
\end{equation*}

%\left|\GL_2(\ZZ/n\ZZ)\right|=\prod_{p^e\:\Vert\:n}p^{4e-3}(p^2-1)(p-1)

\end{proposition}

Note that for a Serre curve, $D\neq 1$ (see \cite[p. 510]{lenstra-stevenhagen-moree}) and thus $K$ is a quadratic field. Also, observe that the above proposition for $g(n)=1$ reduces to the product expression of the Titchmarsh divisor problem for the family of division fields attached to a Serre curve $E$. We note that the product expression for this constant and two other constants corresponding to different $g(n)$'s for such families are given in \cite[Theorem 5]{cojocaru:tdp} by determining the value of $[\mathbb{Q}(E[n]):\mathbb{Q}]$ for a Serre curve $E$ (see \cite[Proposition 17 (iv)]{cojocaru:tdp}) and employing \cite[Lemma 3.12]{kowalski}. It is worth mentioning that a similar approach in finding the expression \eqref{kummer-tdp-lastproduct} using the exact formulas for $[K_n: \mathbb{Q}]$ as given in \cite[Proposition 4.1]{wagstaff} will result in the tedious case by case computations that does not appear to be straightforward.  Especially when $a<0$, this approach seems to be intractable.  The method of \cite{lenstra-stevenhagen-moree} as described above provides an elegant approach to establishing identities similar to \eqref{kummer-tdp-lastproduct} and \eqref{kummer-tdp-lastproduct-2}.

The structure of the paper is as follows. We describe our adaptation of the character sums method of \cite{lenstra-stevenhagen-moree} in Sections 2 and 3 and prove Proposition \ref{character} that plays a crucial role in our explicit computation of the constants in the Kummer case.  Section 4 is dedicated to a proof of Theorem \ref{main1}. The proofs of Theorem \ref{product-kummer-family} and its consequence, Proposition  \ref{TDPK-formula}, are given respectively in Sections 5 and 6. Finally, a brief discussion on Serre curves and the proof of Proposition \ref{prop-Serre} are provided in Section 7. 

\begin{notation}
 The following notations are used throughout the paper. The letter $p$ denotes a prime number,  $k$ denotes a non-negative integer,  the letter $n$ denotes a positive integer,
the multiplicity of the prime $p$ in the prime factorization of $n$ is denoted by $\nu_p(n)$,   the cardinality of a finite set $S$ is denoted by $\#S$, $1_S$ is the characteristic function of a set $S$, $\overline{\mathbb{Q}}$ is an algebraic closure of $\mathbb{Q}$, $\zeta_n$ denotes a primitive root of unity in $\overline{\mathbb{Q}}$, and $\Phi(n)$ is the Euler totient function. In Sections \ref{section:character}, \ref{Sec-3}, \ref{S1}, and  \ref{Section 5}, $a=\pm a_0^e$, with $\sgn(a_0)=\sgn(a)$, is a non-zero integer other than $\pm 1$,  the collection $\{K_n=\mathbb{Q}(\zeta_n, a^{1/n}) \}_{n\in \mathbb{N}}$ is the family of Kummer fields and $K=\mathbb{Q}(a_0^{1/2})$ is the entanglement field attached to this family,  $D$ is the discriminant of $K$, the Galois group of $K_n$ over $\mathbb{Q}$ is denoted by $G(n)$, the inverse limit of the directed family $\{G(n)\}$ is denoted by $G$, $\mu_\infty\subseteq\overline{\mathbb{Q}}$ denotes the group of all roots of unity, and $\mathbb{Q}_{ab}=\mathbb{Q}(\mu_\infty)$ is the maximal abelian extension of $\QQ$. The group of $n$-radicals   of the integer $a=\pm a_0^e$ is denoted by $R_n$ and $R_{\infty}=\bigcup_{n\geq 1}R_n$.
The group of automorphisms of $R_n$ (respectively $R_\infty$) that fix $\mathbb{Q}^\times$ is denoted by $A(n)$ (respectively $A$). The inverse limit of the system 
$\{A(p^k)\}_{k\geq 1}$ is denoted by $A_p$. The map $\pi_{A, n}$ (respectively $\pi_{G, n}$ and $\varphi_{p^k}$)  is the projection map from $A$ (respectively $G$ and $A_p$) to $A(n)$ (respectively $G(n)$ and  $A(p^k)$). The profinite completion of $\mathbb{Z}$ is denoted by $\widehat{\mathbb{Z}}$ and $\mathbb{Z}_p$ is the ring of $p$-adic integers. The normalized Haar measures on $G$, $A$, and $A_p$ are denoted respectively by $\nu_G$, $\nu_A$, and $\nu_{A_p}$.  The space of $\nu$-integrable functions is denoted by $L^1(\nu)$.
In Section \ref{Section 3}, $G(n)$, $A(n)$, $A(p^k)$, $G$, $A$, $A_p$, $\pi_{A, n}$, $\varphi_{p^k}$, $\nu_G$, $\nu_A$, and $\nu_{A_p}$ are used in the general setting of profinite groups.
Finally, in Section \ref{Serre section},  $E[n]$ denotes the group of $n$-division points over $\overline{\mathbb{Q}}$ of an elliptic curve $E$ defined over $\mathbb{Q}$ given by a Weierstrass equation with discriminant $\Delta$, and $K=\mathbb{Q}(\Delta^{1/2})$ of discriminant $D$ is the entanglement field attached to the family of division fields of a Serre elliptic curve. We denote the group of automorphisms of $E[n]$ by $\Aut(E[n])$ and the multiplicative group of $2\times 2$ matrices with entries in $\ZZ/n\ZZ$ by $\GL_2(\ZZ/n\ZZ)$.  \end{notation}

%{\color{red} AUT(E[]) and $GL_2$}

\section{The associated character to a Kummer family}\label{section:character}

Recall that for an integer $a \neq0,\pm1$, we set $a=\pm a_0^e$, where $\sgn(a)=\sgn(a_0)$ and $e$ is the largest such integer. {
We have the group embedding $\mathbb{Z} \simeq \langle a_0 \rangle \subset \overline{\mathbb{Q}}^\times$ defined by sending $1$ to $a_0$. Since $ \overline{\mathbb{Q}}^\times$ is a divisible group, we can extend this embedding to an embedding $\mathbb{Q} \rightarrow \overline{\mathbb{Q}}^\times$. For $a_0$, we fix such embedding $q\mapsto a_0^q$ and write $a_0^\mathbb{Q}$ for the image of this embedding in $\overline{\mathbb{Q}}^\times$. We denote the image of $1/2$ in this embedding by $a_0^{1/2}$. %In addition, we assumed that the choice for $a_0^{1/2}$ is included in the embedding.
%We fix a solution of the equation $x^2-a_0=0$, denote it by $a_0^{1/2}$, and set $K=\QQ(a_0^{1/2})$. 
}

We next
%, corresponding to the quadratic field $K$, 
define a quadratic character which describes the entanglements within a given Kummer family $\{K_n\}$. Let $\mu_{\infty}=\bigcup_{n\geq1}\mu_n(\overline{\mathbb{Q}})$ be the group of all roots of unity in $\overline{\mathbb{Q}}$. Then, $\mu_{\infty}$ is contained in $K_{\infty}=\bigcup_{n\geq1}K_n$. In addition, the infinite extension $K_{\infty}/\QQ$ is the compositum of $\QQ(a_0^{\QQ})$ and $\QQ_{ab}$ (the maximal abelian extension of $\QQ$), where
\begin{equation}
\label{intersection}
\QQ(a_0^{\QQ})\cap\QQ_{ab}=\QQ(a_0^{1/2})
\end{equation}
 (see \cite{lenstra-stevenhagen-moree}*{Lemma 2.5} and the discussion on the last paragraph of \cite{lenstra-stevenhagen-moree}*{p. 494}). The equality \eqref{intersection} demonstrates the entanglement of two fields $\QQ(a_0^{\QQ})$ and $\QQ_{ab}$ since their intersection is a non-trivial extension of $\QQ$ (i.e., the field $\QQ(a_0^{1/2}))$. This justifies calling $K=\QQ(a_0^{1/2})$ the entanglement field.
% Note that $a_0^{\QQ}=\{a_0^b;\;b\in\QQ\}$.  
In \cite{lenstra-stevenhagen-moree}*{p. 494} it is proved that
\begin{equation}
\label{semidirect}
    A=\Aut_{\QQ^{\times}\cap R_{\infty}}(R_{\infty})\cong\Hom(a_0^{\QQ}/a_0^{\ZZ},\mu_{\infty})\rtimes \Aut(\mu_{\infty}),
\end{equation}
where 
% $a_0^{\QQ}=\{a_0^b;\;b\in\QQ\}$ and 
$a_0^{\ZZ}=\{a_0^b;\;b\in\ZZ\}$,
% (see \cite{lenstra-stevenhagen-moree}*{Page 494} for details).
and for $(\phi_1, \sigma_1), (\phi_2, \sigma_2)\in A$ we have
$$(\phi_1, \sigma_1)(\phi_2, \sigma_2)= (\phi_1\cdot (\sigma_1 \circ \phi_2), \sigma_1\circ \sigma_2).$$

Note that $G=\Gal(K_{\infty}/\QQ)$ can be embedded in $A$. Thus, if $(\phi,\sigma)\in \Hom(a_0^{\QQ}/a_0^{\ZZ},\mu_{\infty})\rtimes \Aut(\mu_{\infty})\cong A$ is an element of $G$, then, by \eqref{intersection},  the action of $\phi$ and $\sigma$ on {$a_0^{1/2}$} must be the same. One can show that $(\phi,\sigma)\in A$ is in $G$ if and only if $\phi$ and $\sigma$ act in a compatible way on $a_0^{1/2}$, i.e.,
\begin{equation}
\label{compatibility}
    \phi(a_0^{1/2})=\frac{\sigma(a_0^{1/2})}{a_0^{1/2}}\in \mu_2
\end{equation}
(see \cite{lenstra-stevenhagen-moree}*{p. 494}). (For simplicity, we used $\phi(a_0^{1/2})$ instead of $\phi(a_0^{1/2}a_0^{\mathbb{Z}})$.) We elaborate on \eqref{compatibility} by considering two distinct quadratic characters $\psi_K$ and $\chi_D$ on $A$ which are related to the entanglement field $K=\QQ(a_0^{1/2})$ of discriminant $D$. The quadratic character $\psi_K:A\to\mu_2$ corresponds to the action of $\phi$-component of $(\phi,\sigma)\in A$ on $a_0^{1/2}$, i.e.,
\begin{equation*}
    \psi_K(\phi,\sigma)=\phi(a_0^{1/2})\in\mu_2.
\end{equation*}
This is a \emph{non-cyclotomic character}, i.e., $\psi_K$ does not factor via the natural map $A\to\Aut(\mu_{\infty})$ (see \cite{lenstra-stevenhagen-moree}*{p. 495}). 
The other quadratic character,
\begin{equation*}
    \chi_D:A\to\Aut(\mu_{\infty})\cong\widehat{\ZZ}^{\times}\to\mu_2,
\end{equation*}
corresponds to the action of the cyclotomic component $\Aut(\mu_{\infty})$ of $A$ on $K=\QQ(a_0^{1/2})$ of discriminant $D$, i.e.,
\begin{equation*}
    \chi_D(\phi,\sigma)=\frac{\sigma(a_0^{1/2})}{a_0^{1/2}}\in \mu_2.
\end{equation*}
Hence, by \cite{cox}*{Proposition 5.16 and Corollary 5.17}, $\chi_D$ factors via the lift of the Kronecker symbol $\left(\frac{ D}{.}\right)$ to $\Aut(\mu_{\infty})\cong\widehat{\ZZ}^{\times}$.

The characters $\chi_D$ and $\psi_K$ are not the same on $A$ since one is cyclotomic, and the other is not. Moreover, by \eqref{compatibility}, an element $x\in A$ is in $G$ if and only if $\psi_K(x)=\chi_D(x)$. Thus, if $r: G\rightarrow A$ is the natural embedding defined in \cite[p. 493]{lenstra-stevenhagen-moree}, the image of $r$ is the kernel of the non-trivial quadratic character $\chi=\psi_K\cdot\chi_D: A\to\mu_2$. In other words, the sequence
\begin{equation*}
    1\longrightarrow G\stackrel{r}{\longrightarrow} A\xrightarrow{\chi=\psi_K\cdot\chi_D}\mu_2\longrightarrow 1
\end{equation*}
is an exact sequence (see \cite{lenstra-stevenhagen-moree}*{Theorem 2.9}). 

Let $A(p^k)=\Aut_{\mathbb{Q}^\times \cap R_{p^k} }(R_{p^k})$ and $A_p=\varprojlim A(p^k)$. Since an element of $A$ can be determined by its action on prime power radicals, we have that $A\cong \prod_{p} A_p$  (see \cite[formula (2.10), p. 495]{lenstra-stevenhagen-moree} and \cite[p. 20]{MS}).  
The character $\chi_D$ is the lift of the Kronecker symbol $\left(\frac{D}{.}\right)$ to $A$ via the maps
\begin{equation*}
    A\cong\left(\prod_pA_p\right)\stackrel{\proj}{\longrightarrow}\Aut(\mu_{\infty})\left(\cong\prod_p\ZZ^{\times}_p\right)\stackrel{\text{proj}}{\longrightarrow}(\ZZ/|D|\ZZ)^{\times},
\end{equation*}
where the first projection comes via \eqref{semidirect}.
Since $D$ is a fundamental discriminant, $\chi_D=\prod_{p\mid D}\chi_{D,p}$, where $\chi_{D,p}$ is the lift of the Legendre symbol modulo $p$ to $A_p$ for odd $p$, and $\chi_{D,2}$ is the lift of one of the Dirichlet characters mod $8$ to $A_2$ (see \cite{davenport}*{Chapter 5}). More precisely, if $D$ is odd, then $\chi_{D,2}=1$; if $4~\Vert~D$, then $\chi_{D,2}$ is the lift to $A_2$ of  $\left(\frac{-4}{.}\right)$, the unique Dirichlet character mod $8$ of conductor $4$; and if $8~\Vert~D$, then $\chi_{D,2}$ is the lift to $A_2$ of $\left(\frac{\pm8}{.}\right)$, one of the two Dirichlet characters mod $8$ of conductor $8$. For the case $8~\Vert~D=\pm 2^a\prod_{i=1}^kp_i$, if $D>0$ and the number of $1\leq i\leq k$ with $p_i\equiv3\;(\modd\;4)$ is even, or $D<0$ and the number of $1\leq i\leq k$ with $p_i\equiv3\;(\modd\;4)$ is odd, then $\chi_{D,2}$ is the lift to $A_2$ of $\left(\frac{8}{.}\right)$. Otherwise, $\chi_{D,2}$ is the lift to $A_2$ of $\left(\frac{-8}{.}\right)$. 

Next, we show that $\chi$ can be written as a product of local characters $\chi_p: A_p\to\mu_2$. Note that  $\psi_K$ factors via $A_2$. Let $\psi_{K,2}:A_2\to\mu_2$ be the corresponding homomorphism obtained from factorization of $\psi_K$ via $A_2$. 
For odd primes $p\nmid D$, set $\chi_p=1$. Let $\chi_p=\chi_{D,p}$ for odd primes $p\mid D$ and for prime $2$ let $\chi_2=\chi_{D,2}\cdot\psi_{K,2}$. Therefore, by the above construction of $\chi$, we have the decomposition $\chi=\prod_p\chi_p$.

\section{The local characters $\chi_p$}

\label{Sec-3}

In this section, we find the smallest values of $k$, as a function of $p$ and $a$, for which the local character $\chi_p$ factors via $A(p^k)$. In other words, we will determine the values of $k$ for which $\chi_p$ is trivial on $\ker\varphi_{p^{k}}$ and it is nontrivial on $\ker\varphi_{p^{k-1}}$, where $\varphi_{p^i}$ is the projection map from $A_p$ to $A(p^i)$.    The values are recorded in the statements of Theorem \ref{product-kummer-family} and Proposition \ref{character}. We start 
by giving a concrete description of the groups $A(p^k)$, for positive integers $k$, as  subgroups of matrices 
$\begin{psmallmatrix}1 & 0\\b & d\end{psmallmatrix}$, where $b\in \mathbb{Z}/p^k\mathbb{Z}$ and $d\in \left(\mathbb{Z}/p^k\mathbb{Z}\right)^\times$. We achieve this by choosing a certain compatible system of generators for the groups $R_{p^k}$, where $k\geq 1$.

\begin{proposition}
\label{new-prop}
(i) If $e$ is even and $a<0$, by choosing a suitable set of generators for $R_{2^k}$,  we have that 
\begin{equation*}
A(2^k)\cong \left\{\left(\begin{array}{cc} 1&0\\b&d \end{array} \right);~ b\in \mathbb{Z}/2^k\mathbb{Z}, d\in \left(\mathbb{Z}/2^k\mathbb{Z}\right)^\times, {\rm and}~2b+1\equiv d~({\rm mod}~2^{\min\{k,\nu_2(e)+1\}})
\right\}.
\end{equation*}
(ii) If $p$ is odd, or $p=2$ and $e$ is odd, or $p=2$ and $a>0$, by choosing a suitable set of generators for $R_{p^k}$,  we have that 
\begin{equation*}
A(p^k)\cong \left\{\left(\begin{array}{cc} 1&0\\b&d \end{array} \right);~ b\in \mathbb{Z}/p^k\mathbb{Z}, d\in \left(\mathbb{Z}/p^k\mathbb{Z}\right)^\times, {\rm and}~b+1\equiv d~({\rm mod}~p^{\min\{k,\nu_p(e)\}})
\right\}.
\end{equation*}
(iii) Let $\Phi(n)$ be the Euler totient function and $s$ be as defined in \eqref{s-def}. For odd $p$, $$\#A(p^k)= p^{k-\min\{k, \nu_p(e)\}}\Phi(p^k),$$ and for $p=2$,
$$\#A(2^k)=\begin{cases}
2^{k-\min\{k, s-1\}}\Phi(2^k)&\text{if}~e~\text{is~odd~or}~a>0,\\
2^{k-\min\{k, s-1\}}\Phi(2^{k+1})&\text{if}~e~\text{is~even~and}~a<0.
\end{cases}$$

\end{proposition}
\begin{proof}
%(i) For odd primes $p$, it is known that $A(p^k)\cong G(p^k)={\rm Gal}(K_{p^k}/\mathbb{Q})$ (see \cite[Remarks 2.12. (b), p. 496]{lenstra-stevenhagen-moree}). Also if
%$K\neq \mathbb{Q}(\sqrt{\pm 2})$, then $A(2^k)\cong G(2^k)$. Since the size of $A(2^k)$ is independent of $K$, we get the formulas for the size of $A(p^k)$ from the ones for $G(p^k)$ as given in \cite[Proposition 4.1]{wagstaff}.

(i) Let $a=-a_0^e$ as before and $e=2^{\nu_2(e)}e_1$, where $\nu_2(e)\geq 1$ and $e_1$ is odd. For $q\in \mathbb{Q}$, let $q\mapsto a_0^q$ be the fixed embedding $\mathbb{Q} \rightarrow \overline{\mathbb{Q}}^\times$ defined at the beginning of Section 2. We also fix a collection $\{\zeta_n;~n\in \mathbb{N}\}$ of primitive roots of unity for which $(\zeta_{mn})^{m}=\zeta_n$ for $m, n\in \mathbb{N}$.

%We denote a primitive $m$-th root of unity by $\zeta_m$. 
Recall that $R_{2^k}$ is the group of $2^k$- radicals. We have  $$R_{2^k} =\langle\zeta_{2^{k+1}} \left(a_0^{e_1}\right)^{1/2^{k-\nu_2(e)}}, \zeta_{2^k}  \rangle=\langle \beta, \zeta_{2^k}  \rangle.$$
An automorphism $\tau\in A(2^k)$ is determined by its action on these generators of $R_{2^k}$, i.e., $\beta$ and $\zeta_{2^k}$. We have
$\tau(\beta)=\beta \zeta_{2^{k}}^{b(\tau)}$ and $\tau(\zeta_{2^k})=\zeta_{2^k}^{d(\tau)}$, where $b(\tau)\in \mathbb{Z}/2^k\mathbb{Z}$ and $d(\tau)\in \left(\mathbb{Z}/2^k\mathbb{Z}\right)^{\times}$.  
We consider two cases.

Case 1: $k\geq \nu_2(e)+1$. We have
$$a_0^{e_1} \tau(\zeta_{2^{k+1}}^{2^{k-\nu_2(e)}})=\tau(\beta^{2^{k-\nu_2(e)}})=(\beta \zeta_{2^{k}}^{b(\tau)})^{2^{k-\nu_2(e)}}=a_0^{e_1} \zeta_{2^{k+1}}^{2^{k-\nu_2(e)}}  \zeta_{2^{k}}^{b(\tau)2^{k-\nu_2(e)}}. $$
From here we get 
$$\zeta_{2^k}^{d(\tau)2^{k-\nu_2(e)-1}}=\zeta_{2^k}^{2^{k-\nu_2(e)-1}+b(\tau) 2^{k-\nu_2(e)}}.$$
Since $k\geq \nu_2(e)+1$, this is equivalent to $2b(\tau)+1\equiv d(\tau)~({\rm mod}~2^{\nu_2(e)+1})$. 

Case 2: $k< \nu_2(e)+1$.  We have
$$(a_0^{e_1})^{2/2^{k-\nu_2(e)}} \tau(\zeta_{2^{k+1}}^{2})=\tau(\beta^{2})=(\beta \zeta_{2^{k}}^{b(\tau)})^{2}=(a_0^{e_1})^{2/2^{k-\nu_2(e)}} \zeta_{2^{k+1}}^{2}  \zeta_{2^{k}}^{2b(\tau)}. $$
From here we get 
$$\zeta_{2^k}^{d(\tau)}=\zeta_{2^k}^{1+2b(\tau) }.$$
This is equivalent to $2b(\tau)+1\equiv d(\tau)~({\rm mod}~2^{k})$.

So any $\tau\in A(2^k)$ injects to a matrix $\begin{psmallmatrix}1 & 0\\b(\tau) & d(\tau)\end{psmallmatrix}$ in the group of matrices given in part (i) of the proposition. Thus, $A(2^k)$ is isomorphic to a subgroup of the given group of matrices. We claim that $A(2^k)$ is, in fact, isomorphic to the whole of this group of matrices.
%Since, in each case, the number of such matrices is equal to the cardinality of $A(2^k)$ given in part (i), then the claimed isomorphism in part (ii) holds.
%As it is proved in Proposition 3.1 (ii), $A(2^k)$ is a subgroup of the above group of matrices. 

To prove the claimed isomorphism, it is enough to show that for any $\tau\in A(2^k)$, the congruence
\begin{equation}
\label{only}
2b(\tau)+1\equiv d(\tau)~({\rm mod}~2^{\min\{k,\nu_2(e)+1\}})
\end{equation}
is the only relation between $b(\tau)$ and $d(\tau)$ appearing in $\tau(\beta)=\beta \zeta_{2^{k}}^{b(\tau)}$ and $\tau(\zeta_{2^k})=\zeta_{2^k}^{d(\tau)}$.
%among the entries of such matrices in order to have \eqref{iso}. Observe 
Recall that $R_{2^k}$ is generated by $\beta$ and $\zeta_{2^k}$
%$$R_{2^k} =\langle\zeta_{2^{k+1}} \left(a_0^{e_1}\right)^{1/2^{k-\nu_2(e)}}, \zeta_{2^k}  \rangle=\langle \beta, \zeta_{2^k}  \rangle$$
and the elements of $A(2^k)$ (automorphisms of $R_{2^k}$ that fix $\mathbb{Q}^\times$) are determined by their actions on $\beta$ and $\zeta_{2^k}$. Also, such automorphisms should fix the rational elements of $R_{2^k}$ and preserve any relation between $\beta$ and $\zeta_{2^k}$. Since $R_{2^k}$ is a multiplicative abelian group, any such relationship should be in the form 
\begin{equation}
\label{eqs}
\beta^m \zeta_{2^k}^{n}=r\in \mathbb{Q}^{\times}. 
\end{equation}
To study \eqref{eqs} we consider two cases.

Case (a): Let $k\geq \nu_2(e)+1$ and $\beta^m \zeta_{2^k}^n =r \in \mathbb{Q}^\times$. Then, $|a_0^{me_1/2^{k-\nu_2(e) }}|=|r|\in \mathbb{Q}^\times$. Hence, $m=m_1(2^{k-\nu_2(e) })$, for $m_1 \in \mathbb{Z}$. Replacing this in \eqref{eqs} yields $\zeta_{2^k}^{m_1(2^{k-\nu_2(e)-1 })+n} \in \mathbb{Q}^\times$. Hence, $n=-m_1(2^{k-\nu_2(e)-1})+\ell \Phi(2^k)$, for $\ell \in \mathbb{Z}$. Thus,  the relations \eqref{eqs} are in the form
\begin{equation}
\label{relations}
\beta^{m_1(2^{k-\nu_2(e)})}\zeta_{2^k}^{-m_1(2^{k-\nu_2(e)-1})}=(a_0^{e_1})^{m_1}
\end{equation}
for $m_1 \in \mathbb{Z}$. (Note that $\beta^{2^{k-\nu_2(e)}}\zeta_{2^k}^{-(2^{k-\nu_2(e)-1})}=a_0^{e_1}$ and thus the relations $\beta^{m_1(2^{k-\nu_2(e)})}\zeta_{2^k}^{-m_1(2^{k-\nu_2(e)-1})}=-(a_0^{e_1})^{m_1}$ cannot happen.)
Now if 
%These are equivalent to 
%$$\beta^{2^{k-\nu_2(e)}}= (a_0^{e_1}) \zeta_{2^k}^{2^{k-\nu_2(e)-1}}.$$
$$\beta^{m_1(2^{k-\nu_2(e)})}\zeta_{2^k}^{-m_1(2^{k-\nu_2(e)-1})}=(a_0^{e_1})^{m_1}$$
applying $\tau\in A(2^k)$ on both sides of this identity and following a computation similar to Case 1 above, we conclude that
\begin{equation}
\label{frel}
2b(\tau)+1\equiv d(\tau)~({\rm mod}~2^{\nu_2(e)+1-\nu_2(m_1)})
\end{equation}
for $0\leq \nu_2(m_1) \leq \nu_2(e)$, and no condition if $\nu_2(m_1)>\nu_2(e)$. Since $m_1$ can be any arbitrary integer, then
\begin{equation}
\label{cong1}
2b(\tau)+1\equiv d(\tau)~({\rm mod}~2^{\nu_2(e)+1})
\end{equation}
implies all the congruences \eqref{frel}. 
%Similarly, if 
%These are equivalent to 
%$$\beta^{2^{k-\nu_2(e)}}= (a_0^{e_1}) \zeta_{2^k}^{2^{k-\nu_2(e)-1}}.$$
%$$\beta^{m_1(2^{k-\nu_2(e)})}\zeta_{2^k}^{-m_1(2^{k-\nu_2(e)-1})}=-(a_0^{e_1})^{m_1}$$
%and analogous computation shows that 
%applying $\tau\in A(2^k)$ on both sides of this identity and following a computation similar to Case 1 above, we conclude that
%\begin{equation}
%\label{srel}
%2b(\tau)+1\equiv d(\tau)~({\rm mod}~2^{s-2-\nu_2(m_1)})
%\end{equation}
%for $0\leq \nu_2(m_1) \leq s-3$, and no condition if $\nu_2(m_1)>s-3$. Since $m_1$ can be any arbitrary integer, then
%$$2b(\tau)+1\equiv d(\tau)~({\rm mod}~2^{s-2})$$
%implies all the congruences \eqref{srel}. 
%Thus, the congruence \eqref{cong1} is the only relation between $b(\tau)$ and $d(\tau)$ that appear as a consequence of relations \eqref{relations} between $\beta$ and $\zeta_{2^k}$.

Case (b): Let $k< \nu_2(e)+1$ and $\beta^m \zeta_{2^k}^n  \in \mathbb{Q}^\times$. Then, $\zeta_{2^{k+1}}^m \zeta_{2^k}^n \in \mathbb{Q}^\times$, which implies $\zeta_{2^k}^{\frac{m}{2}+n}=\pm 1$. Hence, $m=2m_1$ for $m_1\in \mathbb{Z}$, and $n=-m_1+\ell \Phi(2^k)$ for $\ell \in \mathbb{Z}$.
Thus,  the relations \eqref{eqs} are in the form
$$\beta^{2m_1}\zeta_{2^k}^{-m_1}=(a_0^{e_1})^{(2m_1)/2^{k-\nu_2(e)}}$$
for $m_1 \in \mathbb{Z}$. (Note that $\beta^{2}\zeta_{2^k}^{-1}=(a_0^{e_1})^{2/2^{k-\nu_2(e)}}$ and thus the relations $\beta^{2m_1}\zeta_{2^k}^{-m_1}=-(a_0^{e_1})^{(2m_1)/2^{k-\nu_2(e)}}$ cannot happen.)
%These are equivalent to 
%$$\beta^{2}= (a_0^{e_1})^{2/2^{k-\nu_2(e)}} \zeta_{2^k}^{}.$$
Applying $\tau\in A(2^k)$ on both sides of this identity and following a computation similar to Case 2 above, we conclude that
%Following the proof of Proposition 3.1 (ii), this relation will result in 
$$2b(\tau)+1\equiv d(\tau)~({\rm mod}~2^{k-\nu_2(m_1)})$$
for $0\leq \nu_2(m_1)\leq k-1$, and no relation if $\nu_2(m_1)>k-1$. Since $m_1$ can be any arbitrary integer, then
\begin{equation}
\label{cong2}
2b(\tau)+1\equiv d(\tau)~({\rm mod}~2^{k})
\end{equation}
implies all these relations.

In conclusion, from \eqref{cong1} and \eqref{cong2}, we get that the congruence \eqref{only}
is the only existing relation among the entries of matrices $\begin{psmallmatrix}1 & 0\\b(\tau) & d(\tau)\end{psmallmatrix}$. Hence, the injection $\tau\rightarrow \begin{psmallmatrix}1 & 0\\b(\tau) & d(\tau)\end{psmallmatrix}$ establishes the claimed isomorphism.

(ii) The proof is analogous to the proof of (ii) by considering compatible systems of generators for the groups $R_{p^k}$'s. More precisely, for $p=2$, $R_{2^k} =\langle \zeta_{2^k}\left(a_0^{e_1}\right)^{1/2^{k-\nu_2(e)}}, \zeta_{2^k}  \rangle,$
where ${\rm gcd}(e_1, 2)=1$, and, for odd $p$, $R_{p^k} =\langle \zeta_{p^k}\left(\pm a_0^{e_1}\right)^{1/p^{k-\nu_p(e)}}, \zeta_{p^k}  \rangle,$
where ${\rm gcd}(e_1, p)=1$.

(iii) This is a consequence of parts (i) and (ii), since 
%It is enough to prove that \begin{equation}
%\label{iso}
%A(2^k)\cong \left\{\left(\begin{array}{cc} 1&0\\b&d \end{array} \right);~ b\in \mathbb{Z}/2^k\mathbb{Z}, d\in \left(\mathbb{Z}/2^k\mathbb{Z}\right)^\times,~ {\rm and}~2b+1\equiv d~({\rm mod}~2^{\min\{k,s-1\}})
%\right\},
%\end{equation}
%because 
the sizes of the groups of matrices in parts (i) and (ii) are the same as the claimed sizes for $A(p^k)$. \end{proof}

The following proposition indicates the significance of the integer $s$ defined in \eqref{s-def}.

\begin{proposition}
\label{factoring}
The number $s$ defined in \eqref{s-def} is the smallest integer $k$ for which  $\psi_{K, 2}$ factors via $A(2^k)$.
%$A(2^{\nu_2(e)+1})$.
\end{proposition}
\begin{proof}
For integers $k\geq 0$, let $\varphi_{p^k}:A_p\to A(p^k)$ be the projection map. It is enough to show that $\psi_{K, 2}$ is non-trivial on 
$\ker \varphi_{2^{s-1}}$ and is trivial on $\ker \varphi_{2^{s}}$.
%$\ker \varphi_{2^{\nu_2(e)}}$ and is trivial on $\ker \varphi_{2^{\nu_2(e)+1}}$. 
We write the proof for the twisted case, where $s=\nu_2(e)+2$. The proof for the untwisted case is similar.

Assume that $a=-(a_0^{e_1})^{2^{\nu_2(e)}}$ as in part (i) of Proposition \ref{new-prop} and assume the compatibility conditions for roots of $a_0$ and roots of unity described there. Considering
$$R_{2^{\nu_2(e)+2}}=\langle\zeta_{2^{\nu_2(e)+3}} \left(a_0^{e_1}\right)^{1/4}, \zeta_{2^{\nu_2(e)+2}}  \rangle,$$ let $\alpha\in A_2$ be such that
$$\tau_2=\varphi_{2^{\nu_2(e)+2}}(\alpha)= \left(\begin{array}{cc} 1&0\\0&1+2^{\nu_2(e)+1} \end{array}\right)\in A(2^{\nu_2(e)+2}).$$
%and 
%$$\tau_1=\varphi_{2^{\nu_2(e)+1}}(\alpha)= \left(\begin{array}{cc} 1&0\\2^{\nu_2(e)}&1 \end{array}\right)\in A(2^{\nu_2(e)+1}).$$
Observe that $\alpha\in \ker \varphi_{2^{\nu_2(e)+1}}$ and  we have
\begin{equation}
\label{T1}
\tau_2(\zeta_{2^{\nu_2(e)+3}} \left(a_0^{e_1}\right)^{1/4})=\zeta_{2^{\nu_2(e)+3}} \left(a_0^{e_1}\right)^{1/4}.
\end{equation}
Raising both sides of \eqref{T1} to power $2$ and observing that
%Now since $\tau_1$ is the restriction of $\tau_2$ to $A(2^{\nu_2(e)+1})$ and 
$\zeta_{2^{\nu_2(e)+2}}$ and $ \left(a_0^{e_1}\right)^{1/2}$ are in $R_{2^{\nu_2(e)+2}}$,  we get
\begin{equation}
\label{T2}
\tau_2(\zeta_{2^{\nu_2(e)+2}}) \tau_2( \left(a_0^{e_1}\right)^{1/2})=\zeta_{2^{\nu_2(e)+2}} \left(a_0^{e_1}\right)^{1/2}.
\end{equation}
Now since 
$\tau_2(\zeta_{2^{\nu_2(e)+2}})= \zeta_{2^{\nu_2(e)+2}}^{1+2^{\nu_2(e)+1}}=-\zeta_{2^{\nu_2(e)+2}}$, the equation \eqref{T2} implies that $$\tau_2( a_0^{e_1/2})=-  a_0^{e_1/2}.$$ 
%We have $e_1=2m+1$ for some integer $m$. 
Hence,
$$\tau_2( a_0^{1/2}a_0^{(e_1-1)/2})=-  a_0^{1/2}a_0^{(e_1-1)/2}.$$ 
Thus for $\alpha\in \ker \varphi_{2^{\nu_2(e)+1}}$, we have $\psi_{K, 2}(\alpha)=-1$. Hence, $\psi_{K, 2}$ is non-trivial on $\ker \varphi_{2^{\nu_2(e)+1}}$.

Next, let $\alpha\in A_2$ be such that $\alpha\in \ker \varphi_{2^{\nu_2(e)+2}}$. Hence,
\begin{equation*}
\label{T3}
\tau_3=\varphi_{2^{\nu_2(e)+3}}(\alpha)= \left(\begin{array}{cc} 1&0\\ b&d \end{array}\right)\in A(2^{\nu_2(e)+3})
\end{equation*}
and 
\begin{equation}
\label{T4}
%\tau_2=\varphi_{2^{\nu_2(e)+2}}(\alpha)
%= \left(\begin{array}{cc} 1&0\\ b&d \end{array}\right)
\left(\begin{array}{cc} 1&0\\ b&d \end{array}\right)
= \left(\begin{array}{cc} 1&0\\ 0&1 \end{array}\right)~{\rm in}~A(2^{\nu_2(e)+2} ).
%({\rm mod}~2^{\nu_2(e)+2} ).
\end{equation}
Hence, $b=2^{\nu_2(e)+2}b_1$ for some integer $b_1$. We have
\begin{equation*}
%\label{T1}
\tau_3(\zeta_{2^{\nu_2(e)+4}} \left(a_0^{e_1}\right)^{1/8})=\zeta_{2^{\nu_2(e)+4}} \left(a_0^{e_1}\right)^{1/8} \zeta_{2^{\nu_2(e)+3}}^{2^{\nu_2(e)+2}b_1}.
\end{equation*}
Squaring both sides of this identity yields
\begin{equation*}
%\label{T1}
\tau_3(\zeta_{2^{\nu_2(e)+3}}) \tau_3(\left(a_0^{e_1}\right)^{1/4})=\zeta_{2^{\nu_2(e)+3}} \left(a_0^{e_1}\right)^{1/4}.
\end{equation*}
This implies
\begin{equation}
\label{T5}
\tau_3(\left(a_0^{e_1}\right)^{1/4})=\frac{\zeta_{2^{\nu_2(e)+3}}}{\zeta_{2^{\nu_2(e)+3}}^d} \left(a_0^{e_1}\right)^{1/4}.
\end{equation}
Now observe, 
from  \eqref{T4}, that
%$b=2^{\nu_2(e)+1}b_1$ for some integer $b_1$ and 
\begin{equation}
\label{T6}
d=1+2^{\nu_2(e)+2}d_1
\end{equation}
for some integer $d_1$. Raising both sides of \eqref{T5} to power 2 and employing \eqref{T6} yield 
$$\tau_3(\left(a_0^{e_1}\right)^{1/2})= \left(a_0^{e_1}\right)^{1/2}.
$$
Hence,
$$\tau_3( a_0^{1/2}a_0^{(e_1-1)/2})= a_0^{1/2}a_0^{(e_1-1)/2}.$$ 
%where $e_1=2m+1$.
Thus, $\psi_{K, 2}$ is trivial on $\ker \varphi_{2^{\nu_2(e)+2}}$.

\end{proof}

The following proposition is essential in proving Theorem \ref{product-kummer-family}.

\begin{proposition}
\label{character}
Let $\ell(p)$ be the smallest integer $k$ for which $\chi_p$ factors via $A(p^{k})$. Then
$$\ell(p)=\left\{\begin{array}{ll}
0&\text{if~} p~ \text{is~ odd~and~}p\nmid D,\\1&\text{if~} p~ \text{is~ odd~and~}p\mid D,\\
s&\text{if}~ p=2~ \text{and}~ D ~\text{is odd},\\
\max\{2,s\}&\text{if}~ p=2~ \text{and}~ 4\Vert D,\\

2&\text{if}~ p=2,~8\Vert D,~\text{and}~(\nu_2(e)=1~\text{and}~a<0),\\  
\max\{3,s\}&\text{if}~ p=2,~ 8\Vert D,~\text{and}~(\nu_2(e)\neq 1~\text{or}~a>0).\end{array}
\right.
$$

\end{proposition}

\begin{proof}
If $p\nmid 2D$, by the definition of $\chi_p$, we have that $\chi_p$ is constantly equal to $1$. Thus, the assertion holds.

If $p$ is an odd integer dividing $D$, then $\chi_p$ is the Legendre symbol mod $p$, so the result follows.

If $p=2$ and $D$ is odd, then $\chi_2=\psi_{K, 2}$. Thus, the result follows from Proposition \ref{factoring}.

If $p=2$ and $4\Vert D$, then $\chi_2=\psi_{K, 2} \chi_{D, 2}$, where $\chi_{D, 2}$ is the Dirichlet character mod $8$ of conductor $4$. 
{We are looking for a positive integer $k$  such that $\psi_{K,2} (\alpha)\neq \chi_{D, 2}(\alpha)$ for an element $\alpha\in \ker\varphi_{2^{k-1}}$, and $\psi_{K,2} (\alpha)= \chi_{D, 2}(\alpha)$ for all $\alpha\in \ker\varphi_{2^k}$.}
Note that $2$ is the smallest value of $k$ for which $\chi_{D, 2}$ factors via $A(2^k)$, and, by Proposition \ref{factoring}, $s$ is the smallest value of $k$ for which $\psi_{K, 2}$ factors via $A(2^k)$. Thus, $\chi_{D, 2}$ is trivial on $\ker\varphi_{2^k}$ for $k\geq 2$ and is nontrivial on $\ker\varphi_{2^{k}}$ for $0\leq k\leq 1$. Also, $\psi_{K, 2}$ is trivial on $\ker\varphi_{2^k}$ for $k\geq s$ and is nontrivial on $\ker\varphi_{2^{k}}$ for $0\leq k< s$. Using these facts and a case-by-case analysis in terms of the values of $\nu_2(e)$,  and for the untwisted and twisted cases, 
we can see that the claimed assertion, in this case, holds.  More precisely, if $\nu_2(e)=0$, then $\chi_2$ factors via $A(2^2)$. Otherwise, $\chi_2$ factors via $A(2^s)$.
The only case that needs special attention is when $s=2$, i.e.,  $a$ is an exact perfect square (i.e., $a>0$ and $\nu_2(e)=1$). In this case, $\max\{2, s\}=2$ and both $\psi_{K, 2}$ and $\chi_{D, 2}$ are trivial on $\ker\varphi_{2^2}$, hence $\chi_2$ is trivial on $\ker\varphi_{2^2}$. Let $\alpha\in \ker\varphi_{2}$ be such that $\varphi_{2^2}(\alpha)= \begin{psmallmatrix}1 & 0\\0 & 3\end{psmallmatrix}\in A(2^2)$. Note that $0+1\equiv 3$ (mod $2$), so by Proposition \ref{new-prop}(ii) such $\alpha$ exists. We have $\chi_2(\alpha)=\psi_{K, 2}(\alpha)\chi_{D, 2}(\alpha)=(1)(-1)=-1$. Thus, $\chi_2$ is non-trivial  on $\ker\varphi_{2}$. Hence, $\chi_2$ factors via $A(2^2)=A(2^s)=A(2^{\max\{2, s\}})$ but not via $A(2)$.

If $p=2$ and $8\Vert D$, similar to part (iv), a case-by-case analysis in terms of the values of $\nu_2(e)$, and for the untwisted and twisted cases, we can verify the result. (Note that in this case $3$ is the smallest values of $k$ for which $\chi_{D, 2}$ factors via $A(2^k)$.) More precisely, if $\nu_2(e)=0$ or $1$, and $-a$ is not a perfect square, then $\chi_2$ factors via $A(2^3)$.  Also, if $\nu_2(e)=1$ and $a<0$, then $\chi_2$ factors via $A(2^2)$.   Otherwise, $\chi_2$ factors via $A(2^s)$. Two cases need special attention.

Case 1: The number $a$ is negative of an exact perfect square (i.e., $a<0$ and $\nu_2(e)=1$). In this case, $\max\{3, s\}=3$ and both $\psi_{K, 2}$ and $\chi_{D, 2}$ are trivial on $\ker\varphi_{2^3}$, hence $\chi_2$ is trivial on $\ker\varphi_{2^3}$. Thus, $\chi_2$ acts through $A(2^3)$. Let $\alpha\in \ker\varphi_{2^2}$. Then $\varphi_{2^3}(\alpha)= \begin{psmallmatrix}1 & 0\\b & d\end{psmallmatrix}\in A(2^3)$ is such that $\begin{psmallmatrix}1 & 0\\b & d\end{psmallmatrix}\equiv \begin{psmallmatrix}1 & 0\\0 & 1\end{psmallmatrix}$ (mod $2^2$). 
Hence,

$$ \left(\begin{array}{cc}1 & 0\\b & d\end{array}\right )\in \left\{ \left(\begin{array}{cc}1 & 0\\0 & 1\end{array}\right ),  \left(\begin{array}{cc}1 & 0\\0 & 5\end{array}\right ),  \left(\begin{array}{cc}1 & 0\\4 & 1\end{array}\right ),  \left(\begin{array}{cc}1 & 0\\4 & 5\end{array}\right )  \right\}\subset A(2^3).$$
Since for each $\alpha$ corresponding to the above matrices we have $\chi_2(\alpha)=\psi_{K, 2}(\alpha)\chi_{D, 2}(\alpha)=1$, we conclude that $\chi_2$ is trivial on $\ker \varphi_{2^2}$. Now let $\alpha\in \ker\varphi_{2}$ be such that $\varphi_{2^3}(\alpha)= \begin{psmallmatrix}1 & 0\\6 & 1\end{psmallmatrix}\in A(2^3)$. Note that $(2)(6)+1\equiv 1$ (mod $2^2$), so by Proposition \ref{new-prop}(i) such $\alpha$ exists. We have $\chi_2(\alpha)=\psi_{K, 2}(\alpha)\chi_{D, 2}(\alpha)=(-1)(1)=-1$. Thus, $\chi_2$ is non-trivial  on $\ker\varphi_{2}$. Hence, $\chi_2$ factors via $A(2^2)$ as claimed.

Case 2:  The number $a$ is an exact perfect fourth power (i.e., $a>0$ and $\nu_2(e)=2$). In this case, $\max\{3, s\}=3$ and both $\psi_{K, 2}$ and $\chi_{D, 2}$ are trivial on $\ker\varphi_{2^3}$, hence $\chi_D$ is trivial on $\ker\varphi_{2^3}$. Let $\alpha\in \ker\varphi_{2^2}$ be such that $\varphi_{2^3}(\alpha)= \begin{psmallmatrix}1 & 0\\0 & 5\end{psmallmatrix}\in A(2^3)$. Note that $0+1\equiv 5$ (mod $2^2$), so by Proposition \ref{new-prop}(ii) such $\alpha$ exists. We have $\chi_2(\alpha)=\psi_{K, 2}(\alpha)\chi_{D, 2}(\alpha)=(1)(-1)=-1$. Thus, $\chi_2$ is non-trivial  on $\ker\varphi_{2^2}$. Hence, $\chi_2$ factors via $A(2^3)=A(2^s)=A(2^{\max\{3, s\}})$.
\end{proof}

\section{Proof of Theorem \ref{main1}}
\label{sec-main1}

%
%Let $(G(n), i_{n,m})_{n,m\in \NN}$ and $(A(n),j_{n,m})_{n,m\in \NN}$ be inverse systems of finite groups with respect to divisibility relation on $\mathbb{N}$.
%Moreover, assume that there is an injective map $r_n:G(n)\to A(n)$ compatible by $i_{n,m}$ and $j_{n,m}$ for all $n\geq 1$, i.e., the diagram 
%\begin{equation*}
%    \begin{tikzcd}
%    G(m)\arrow[r,"i_{n,m}"]\arrow[d,"r_m"] & G(n)\arrow[d,"r_n"]\\
%    A(m)\arrow[r,"j_{n,m}"] & A(n)
%    \end{tikzcd}
%\end{equation*}
%commutes. This yields an injective homomorphism $r: G\to A$ of profinite groups by taking the inverse limit over $n$. Let $\nu_G$ be the normalized Haar measure on the profinite group $G$, and $\nu_A$ be the normalized Haar measure on the profinite group $A$.
%

%This yields an injective map $r:G\to A$ by taking inverse limit over $n$.

\begin{proof}[Proof of Theorem \ref{main1}]
\label{Section 3}
Let $\nu_G$ be the normalized Haar measure on the profinite group $G$, and $\nu_A$ be the normalized Haar measure on the profinite group $A$.
We start by writing the summation 
\begin{equation*}
    \sum_{n\geq1}\frac{g(n)}{\#G(n)}
\end{equation*}
in terms of measures of certain measurable subgroups of $G$. For this purpose, let $\pi_{G,n}: G\to G(n)$ be the projection map for each $n\geq1$. Then, $G/\ker\pi_{G,n}\cong G(n)$ and $[G:\ker\pi_{G,n}]=\#G(n)$. Hence, since $\ker\pi_{G, n}$ is a closed subgroup of $G$, we have $\nu_G(\ker\pi_{G,n})=1/\#G(n)$ (see \cite[Lemma 18.1.1.(a)]{FA}).
%{\color{red}}
%, where $\nu_G$ is the normalized Haar measure on $G$. 
Thus,
\begin{equation}
\label{ev2}
    \sum_{n\geq1}\frac{g(n)}{\#G(n)}=\sum_{n\geq1}g(n)\nu_G(\ker\pi_{G,n}).
\end{equation}
Observe that 
\begin{equation}
\label{index}
[A:r(\ker\pi_{G,n})]=[A:r(G)][r(G):r(\ker\pi_{G,n})].
\end{equation}

%Observe that the index 
%%number of cosets of the set 
%$[A:r(\ker\pi_{G,n})]$ divided by 
%%the number of cosets of the group 
%$[G:\ker\pi_{G,n}]=[r(G):r(\ker\pi_{G,n})]$ is equal to $\#\left( A/r(G)\right)$. 
Also, since $\chi$ is continuous and ${\rm ker}\chi=r(G)$, then $r(G)$ is a closed subgroup of $A$ and hence it is $\nu_A$-measurable. Similarly, since $r({\rm ker} \pi_{G, n})$ is a closed subgroup of $r(G)$ and $r(G)$ is a closed subgroup of $A$, then $r({\rm ker} \pi_{G, n})$ is a closed subgroup of $A$ and hence it is $\nu_A$-measurable. 
Thus, from \eqref{index}, we have 
\begin{equation}
\label{ev3}
    \nu_G(\ker\pi_{G,n})=\frac{\nu_A(r(\ker\pi_{G,n}))}{\nu_A(r(G))}.
\end{equation}
Now, since
\begin{equation}
\label{star}
    1\to G\stackrel{r}{\longrightarrow} A\stackrel{\chi}{\longrightarrow} \mu_m\to 1
\end{equation}
is an exact sequence, by (\ref{ev3}), we have
\begin{equation}
\label{ev4}
    \begin{split}
        \sum_{n\geq1}g(n)\nu_G(\ker\pi_{G,n}) & =\sum_{n\geq1}g(n)\frac{\nu_A(r(\ker\pi_{G,n}))}{\nu_A(\ker\chi)}\\
        & =\frac{1}{\nu_A(\ker\chi)}\sum_{n\geq1}g(n)\nu_A(r(\ker\pi_{G,n})).
    \end{split}
\end{equation}

Next, we show that $r(\ker\pi_{G,n})=\ker(\pi_{A,n})\cap\ker\chi$, where $\pi_{A,n}:A\to A(n)$ is the projection map for each $n\geq1$. To prove this claim, we note that the diagram
\begin{equation}
\label{commutative-G-to-A-main1}
    \begin{tikzcd}
     G\arrow[d,"r"]\arrow[r,"\pi_{G,n}"]&G(n)\arrow[d,"r_n"]\\
     A\arrow[r,"\pi_{A,n}"]&A(n)
    \end{tikzcd}
\end{equation}
commutes. For a group $H$, let $e_H$ denote its identity element. Note that if $\sigma\in\ker\pi_{G,n}$, then $r_n(\pi_{G,n}(\sigma))=r_n(e_{G(n)})=e_{A(n)}$. Hence, by the commutative diagram \eqref{commutative-G-to-A-main1}, we have $r(\sigma)\in\ker(\pi_{A,n})$. Moreover, by the exact sequence \eqref{star}, we have $r(\sigma)\in r(G)=\ker\chi$. Therefore, 
\begin{equation}
\label{ev5-1}
r(\ker\pi_{G,n})\subset \ker(\pi_{A,n})\cap\ker\chi.
\end{equation} 
On the other hand, if $\alpha\in\ker(\pi_{A,n})\cap\ker\chi\subset\ker\chi=r(G)$, then there exists a $\sigma\in G$ such that $r(\sigma)=\alpha$. Moreover, $r(\sigma)\in\ker(\pi_{A,n})$ means $\pi_{A,n}(r(\sigma))=e_{A(n)}$. Hence, $r_n(\pi_{G,n}(\sigma))=e_{A(n)}$ as \eqref{commutative-G-to-A-main1} is commutative. Thus, $\sigma\in\ker\pi_{G,n}$, since $r_n$ is injective. This shows that 
\begin{equation}
\label{ev5-2}
\ker(\pi_{A,n})\cap\ker\chi\subset r(\ker\pi_{G,n}).
\end{equation} 
Therefore, from \eqref{ev5-1} and \eqref{ev5-2}, we have
\begin{equation}
\label{ev5}
    r(\ker\pi_{G,n})=\ker(\pi_{A,n})\cap\ker\chi.
\end{equation}    

From (\ref{ev5}), we have
\begin{equation}
\label{ev6}
    \begin{split}
    \sum_{n\geq1}g(n)\nu_A(r(\ker\pi_{G,n})) 
    & =\sum_{n\geq1}g(n)\nu_A(\ker\pi_{A,n}\cap\ker\chi)\\ 
    & =\sum_{n\geq1}g(n)\int_A1_{\ker\pi_{A,n}\cap\ker\chi}d\nu_A\\ 
    & =\int_A\left(\sum_{n\geq1}g(n)1_{\ker\pi_{A,n}}\right)1_{\ker\chi}d\nu_A.
    \end{split}
\end{equation}
To justify the interchange of the summation and the integral in the last equality, observe that
\begin{equation*}
\left\lvert\sum_{n=1}^mg(n)1_{\ker\pi_{A,n}\cap\ker\chi}\right\rvert\leq\sum_{n\geq1}\lvert g(n)\rvert1_{\ker\pi_{A,n}\cap\ker\chi}.    
\end{equation*} 
Since by the assumption, $\sum_{n\geq1}\lvert g(n)\rvert/\#G(n)$ converges, then,
by \cite[Theorem 1.27]{rudin}, \\$\sum_{n\geq1}\lvert g(n)\rvert1_{\ker\pi_{A,n}\cap\ker\chi}$ is integrable. Thus, by Lebesgue's dominated convergence theorem (see \cite[ Theorem 1.34]{rudin}), the interchange of the summation and the integral in \eqref{ev6} is justified. Also, since $\#A(n)\geq \#G(n)$, then  $\sum_{n\geq1}\lvert g(n)\rvert/\#A(n)<\infty$. Hence, by \cite[Theorem 1.38]{rudin},  
$$\Tilde{g}=\sum_{n\geq1}g(n)1_{\ker\pi_{A,n}}\in L^1(\nu_A).$$
Now from (\ref{ev2}), (\ref{ev4}), and (\ref{ev6}), we have
\begin{equation}
    \label{ev7}
    \sum_{n\geq1}\frac{g(n)}{\#G(n)}
    =\frac{\int_A \Tilde{g}1_{\ker\chi}d\nu_A}{\int_A1_{\ker\chi}d\nu_A}.
\end{equation}
%where $$\Tilde{g}=\sum_{n\geq1}g(n)1_{\ker\pi_{A,n}}.$$

Note that the character $\chi:A\to\mu_m$ in (\ref{star}) induces the character $\chi':A/r(G)\stackrel{\sim}{\longrightarrow}\mu_m$ by $\chi'(\bar{\alpha})=\chi(\alpha)$, where $\alpha\in A$ and $\bar{\alpha}$ is the coset associated to $\alpha$ in $A/r(G)$. In other words, $\chi$ is the lift of $\chi'$ to $A$. Since $\chi'$ sends a generator of $A/r(G)$ to a generator of $\mu_m$, then $\chi'$ is a generator of the group of characters of $A/r(G)$ denoted by $\widehat{A/r(G)}$. Thus, for $\bar{\alpha}\in A/r(G)$, by \cite{course-in-arithmetic}*{Chapter \textrm{VI}, Proposition 4}, we have
\begin{equation*}
    \sum_{\epsilon\in\widehat{A/r(G)}}\epsilon(\bar{\alpha})=\sum_{i=0}^{m-1}(\chi')^i(\bar{\alpha})=
    \begin{cases}
    m &\quad \text{if }\;\bar{\alpha}=e_{A/r(G)},\\
    0 &\quad \text{if }\;\bar{\alpha}\neq e_{A/r(G)}.
    \end{cases}
\end{equation*}
Therefore, since $\bar{\alpha}=e_{A/r(G)}$ means $\alpha\in\ker\chi$, we have
\begin{equation*}
    \sum_{i=0}^{m-1}\chi^i(\alpha)=
    \begin{cases}
    m &\quad \text{if }\;\alpha\in\ker\chi,\\
    0 &\quad \text{if }\;\alpha\notin\ker\chi.
    \end{cases}
\end{equation*}
This implies $\sum_{i=0}^{m-1}\chi^i(\alpha)=m\cdot1_{\ker\chi}(\alpha)$. Thus,
\begin{equation}
\label{ev8}
    \frac{\int_A\Tilde{g}\1_{\ker\chi}d\nu_A}{\int_A\1_{\ker\chi}d\nu_A}=\frac{\int_A\Tilde{g}\sum_{i=0}^{m-1}\chi^id\nu_A}{m\int_A 1_{\ker\chi}d\nu_A}.
\end{equation}
Furthermore, by \eqref{star}, we have $[A:\ker\chi]=[A:r(G)]=m$. Hence, $\nu_A(\ker\chi)=1/m$. Thus, the desired result follows from \eqref{ev7} and \eqref{ev8}.
\end{proof}

The following corollary considers a special case of Theorem \ref{main1}.

\begin{corollary}
\label{cor-after-main1}
In Theorem \ref{main1}, suppose that $g$ is a multiplicative arithmetic function. Assume that $(h_n)_{n\in \mathbb{N}}$ is a family of isomorphisms $h_n: A(n) \rightarrow \prod_{p^k\Vert n} A(p^k)$, compatible with transition maps $(j_{n_1, n_2})_{n_1 \mid n_2}$, that results in a topological isomorphism $A\cong\prod_pA_p$, where $A_p=\varprojlim A(p^i)$. In addition, let $\chi=\prod_p\chi_p$, where the continuous characters $\chi_p:A_p\to\mu_m$ are trivial except for finitely many $p$'s. Then, \begin{equation*}
\label{tilde-g-p}
\Tilde{g}_p=\sum_{k\geq0}g(p^k)1_{\ker\varphi_{p^k}}\in L^1(\nu_{A_p})
\end{equation*}
and
\begin{equation}
\label{new-cor-12}
    \sum_{n\geq1}\frac{g(n)}{\#G(n)}=\sum_{i=0}^{m-1}\prod_p\int_{A_p}\Tilde{g}_p\chi^i_pd\nu_{A_p},
\end{equation}
where $\nu_{A_p}$ is the normalized Haar measure on $A_p$.
\end{corollary}
\begin{proof}
By Theorem \ref{main1}, we have
\begin{equation}
\label{new-cor-1}
    \sum_{n\geq1}\frac{g(n)}{\#G(n)}=\sum_{i=0}^{m-1}\int_A\Tilde{g}\chi^id\nu_A.
\end{equation}
Since $g(n)$ is multiplicative, $A\cong\prod_pA_p$, $\nu_A=\prod_p\nu_{A_p}$, $\chi=\prod_p\chi_p$, and $\Tilde{g}=\prod_p\Tilde{g}_p$, then  \eqref{new-cor-1} yields \eqref{new-cor-12}. Note that, since $\#A(p^k)\geq \#G(p^k)$, then  $\sum_{k\geq 0}\lvert g(p^k)\rvert/\#A(p^k)<\infty$. Hence, by \cite[Theorem 1.38]{rudin},  
$\Tilde{g}_p\in L^1(\nu_{A_p}).$ Thus, the integrals in \eqref{new-cor-12} are finite.
\end{proof}

\section{Proof of Theorem \ref{product-kummer-family}}
\label{S1}

\begin{proof}[Proof of Theorem \ref{product-kummer-family}]
Let the profinite group $A$ and the character $\chi$ be as defined in Section 2. We employ Corollary \ref{cor-after-main1} and compute
$\int_{A_p}\Tilde{g}_pd\nu_{A_p}$ and
$\int_{A_p}\Tilde{g}_p\chi_pd\nu_{A_p}$ for primes $p$.
Since $\ker\varphi_{p^k}$ is a closed subgroup of $A_p$, we have $\nu_{A_p}(\ker\varphi_{p^k})=1/[A_{p}:\ker\varphi_{p^k}]=1/\#A(p^k)$. Observe that
\begin{equation}
\label{sum-in-last-form-1}
    \begin{split}
        \int_{A_p} \Tilde{g}_pd\nu_{A_p}
        & =\int_{A_p}\sum_{k\geq0} g(p^k)1_{\ker\varphi_{p^k}}d\nu_{A_p}\\
        & =\sum_{k\geq0}g(p^k)\nu_{A_p}(\ker\varphi_{p^k})\\
        & =\sum_{k\geq0}\frac{g(p^k)}{\#A(p^k)}.
    \end{split}
\end{equation}

Observe that if $S\subset A_p$ is $\nu_A$-measurable, then
$\int_{S} \chi_p(\alpha x) d\nu_{A_p}(x)= \int_{S} \chi_p(x) d\nu_{A_p}(x)$ for any $\alpha\in S$. From here we conclude that if $\chi_p$ is non-trivial on $S$, then $\int_S \chi_p d\nu_{A_p}=0$. 
Hence, by Proposition \ref{character}, 
\begin{equation}
\label{sum-in-last-form-3}
    \begin{split}
        \int_{A_p}\Tilde{g}_p\chi_pd\nu_{A_p}
        & =\int_{A_p}\left(1_{A_p}\chi_p+g(p)1_{\ker\varphi_p}\chi_p+\dots+g(p^k)1_{\ker\varphi_{p^k}}\chi_p+\dots\right)d\nu_{A_p}\\
        & =0+\sum_{k\geq \ell(p)}g(p^k)\nu_{A_p}(\ker\varphi_{p^k})\\
        & =\sum_{k\geq \ell(p)}\frac{g(p^k)}{\#A(p^k)}.
    \end{split}
\end{equation}

Thus, by Corollary \ref{cor-after-main1} with $m=2$, \eqref{sum-in-last-form-1}, and \eqref{sum-in-last-form-3}, we get \eqref{product-kummer}. 
\end{proof}
\section{Proof of Proposition \ref{TDPK-formula}}
\label{Section 5}
%In \eqref{kummer-sum}, if we set $g=1$, by \eqref{felix-murty-tdp}, we 
\begin{proof}[Proof of Proposition \ref{TDPK-formula}]

For integer $k\geq 1$ and odd prime $p$, let
\begin{equation}
\label{k'}
 k'=
    \begin{cases}
    0 &\text{if }k\leq \nu_p({e}),\\
    k-\nu_p(e)&\text{if }k>\nu_p(e),\\
    \end{cases}
\end{equation}
and for $k\geq 1$ and $p=2$, let
\begin{equation}
\label{k-prime}
 k'=
    \begin{cases}
    0 &\text{if }k\leq \nu_2({e})~\text{and}~ (a>0~ \text{or}~ e~\text{is~odd}),\\
      1 &\text{if }k\leq \nu_2({e})~\text{and}~(a<0~ \text{and}~ e~\text{is~even}),\\
        k-\nu_2(e)&\text{if }k>\nu_2(e).
    \end{cases}
\end{equation}
Then, from Proposition \ref{new-prop} (iii), we have
\begin{equation}
\label{Apk}
\#A(p^k)=\begin{cases}
%p^{k'}\phi(p^k)=
p^{k+k'-1}(p-1)&\text{if }k\geq 1,\\
1&\text{if }k=0.
\end{cases}
\end{equation}

Now by employing \eqref{Apk} in \eqref{product-kummer} we get
\begin{equation}
    \label{kummer-sum}
    \sum_{n=1}^{\infty}\frac{g(n)}{\#G(n)}=\prod_p\left(1+\sum_{k\geq 1}\frac{g(p^k)}{p^{k+k'-1}(p-1)}\right)+\prod_{p}\sum_{k\geq\ell(p)}\frac{g(p^k)}{\#A(p^k)}.
    %p^{k+k'-1}(p-1)}.
    %\prod_p\left(1+\sum_{k\geq1}\frac{g(p^k)}{p^{k+k'-1}(p-1)}\right).
\end{equation}

We set $g=1$ in \eqref{kummer-sum} to
get the product expression for the constant in the conjectured asymptotic formula in the Titchmarsh Divisor Problem for a given Kummer family. Therefore, by \eqref{kummer-sum},
\begin{equation}
\label{expression}
    \sum_{n\geq1}\frac{1}{[K_n:\QQ]}=\left(1+\prod_{p\mid 2D}\frac{C_p}{1+B_p}\right)\prod_p\left(1+B_p\right),
\end{equation}
for the following values for $B_p$ and $C_p$.

If $p$ is odd, we have
\begin{equation}
\label{bp}
    B_p= \sum_{k\geq1}\frac{1}{p^{k+k'-1}(p-1)}=\frac{p^{\nu_p(e)+2}+p^{\nu_p(e)+1}-p^2}{p^{\nu_p(e)}(p-1)(p^2-1)},
\end{equation}
%and for odd primes $p$ we have 
and $C_p=B_p$, where $k^\prime$ is given by \eqref{k'}.

For $p=2$, we have the following cases for $B_2$ and  $C_2$ with $k^\prime$ as given by \eqref{k-prime}.

\noindent {\it Case (i).} Let $e$ be odd or $a>0$. Hence, $s=\nu_2(e)+1$. Then $B_2$ is the same as \eqref{bp} with $p=2$. Now, if $D$ is odd; or $4 \Vert D$ and $s\geq2$; or $8 \Vert D$ and $s\geq3$, then
\begin{equation*}
    C_2= \sum_{k\geq \ell(2)}\frac{1}{2^{k+k'-1}}=\frac{2}{2^{\nu_2(e)}(2^2-1)}.
\end{equation*}
Otherwise,
\begin{equation*}
    C_2=\sum_{k\geq \ell(2)}\frac{1}{2^{k+k'-1}}=\frac{2^{\nu_2(e)+1}}{2^{\beta}(2^2-1)},
\end{equation*}
where $\beta=2$ if $4 \Vert D$ and $s=1$; and $\beta=4$ if $8 \Vert D$ and $s\in \{1, 2\}$. 

\noindent {\it Case (ii).} Let $e$ be even and $a<0$.  Then
\begin{equation*}
    B_2= \sum_{k\geq 1}\frac{1}{2^{k+k'-1}}=\frac{2^{\nu_2(e)+2}-2^{\nu_2(e)}-1}{2^{\nu_2(e)}(2^2-1)}.
\end{equation*}
If $8\Vert D$ and $\nu_2(e)= 1$, we have $\ell(2)=2$. Hence,
$$C_2= \sum_{k\geq \ell(2)}\frac{1}{2^{k+k'-1}}= \frac{1}{2^2-1}.$$
Otherwise, we have $\ell(2)=s=\nu_2(e)+2$ and thus
\begin{equation*}
    C_2= \sum_{k\geq \ell(2)}\frac{1}{2^{k+k'-1}}=\frac{1}{2^{\nu_2(e)+1}(2^2-1)}.    
\end{equation*}

By applying the above expressions in \eqref{expression} and by case-by-case simplifying, we get \eqref{kummer-tdp-lastproduct}.
\end{proof}

\section{Serre Curves}
\label{Serre section}
Let $E$ be an elliptic curve defined over $\QQ$. 
%given by a Weierstrass equation
%\begin{equation*}
%    y^2=x^3+ax+b,
%\end{equation*}
%where $a,b\in\QQ$. 
Let $\QQ(E[n])$ be the $n$-division field of $E$. By taking the inverse limit of the natural injective maps
\begin{equation*}
    r_n:\Gal(\mathbb{Q}(E[n])/\QQ)\to\Aut(E[n])\cong \GL_2(\ZZ/n\ZZ),
\end{equation*}
over all $n\geq1$, we have an injective profinite homomorphism
\begin{equation*}
    r\colon \Gal(\mathbb{Q}(E[\infty])/\QQ)\to\Aut(E[\infty])\cong\GL_2(\widehat{\ZZ}),
\end{equation*}
%Let $\Delta$ be the discriminant of the cubic equation $x^3+ax+b=0$. 
where $E[\infty]=\bigcup_{n=1}^{\infty} E[n]$. Let $\Delta$ be the discriminant of any Weierstrass model for $E$.
Set $K=\QQ({\Delta}^{1/2})$ and let
$D$ be the discriminant of $K$. In anticipation of applying Theorem \ref{main1}, let $\det$ be the determinant map $\det:\GL_2(\widehat{\ZZ})\to\widehat{\ZZ}^{\times}$ and 
\begin{equation*}
    \chi_D:\GL_2(\widehat{\ZZ})\stackrel{\det}{\longrightarrow}\widehat{\ZZ}^{\times}\stackrel{\left(\frac{ D}{.}\right)}{\longrightarrow}\mu_2
\end{equation*}
be the composition of $\det$ with the lift to $\widehat{\ZZ}^{\times}$ of the Kronecker symbol attached to $D$. We note that $\GL_2(\ZZ/2\ZZ)\cong S_3$, where $S_3$ is the symmetric group on three letters. Let 
\begin{equation*}
    \psi:\GL_2(\widehat{\ZZ})\to \GL_2(\ZZ/2\ZZ)\cong S_3\stackrel{\sgn}{\longrightarrow}\mu_2
\end{equation*}
be the composition of the projection map from $\GL_2(\widehat{\ZZ})$ to $\GL_2(\ZZ/2\ZZ)$ with the signature character on $S_3$. Let $G=\Gal(\mathbb{Q}(E[\infty])/\QQ)$. For $\eta \in G$ we can show that the image of $r(\eta)$ under $\psi$ is the same as ${\chi_D(r(\eta))=\eta({\Delta}^{1/2})}/{{\Delta}^{1/2}}$ (see (26) in \cite{cojocaru:tdp} and discussion before it). We now set $\chi=\chi_D\cdot\psi$. 

The above construction of the character $\chi$ is described by J.-P. Serre in \cite{serre}.  In addition,  in \cite[Section 5.5]{serre},
Serre shows that the character $\chi$ constructed above is non-trivial 
and $r(G)$ is contained in $\ker\chi$, hence
%, for $E$ without complex multiplication,  
$[\GL_2(\widehat{\ZZ}):r(G)]\geq2$. We name $E$ a \emph{Serre curve} if
$[\GL_2(\widehat{\ZZ}):r(G)]=2$. This is equivalent to saying that $r(G)=\ker\chi$. Thus, letting $A=\GL_2(\widehat{\ZZ})$, for Serre curve $E$, the sequence
\begin{equation}
\label{exactec}
1 \longrightarrow G \stackrel{r}{\longrightarrow} A \stackrel{\chi}{\longrightarrow} \mu_{2} \longrightarrow 1
\end{equation}
is an exact sequence. In addition for a Serre curve $K=\mathbb{Q}(\Delta^{1/2})$ is a quadratic field (see the first paragraph of \cite[p. 510]{lenstra-stevenhagen-moree} for explanation).

The quadratic character $\chi:\GL_2(\hat{\ZZ})(\cong\prod_p\GL_2(\ZZ_p))\to\mu_2$ can be written as a product of local characters $\chi_p:\GL_2(\ZZ_p)\to\mu_2$. Observe that since $\psi$ factors via $\GL_2(\ZZ/2\ZZ)$, then it factors via $\GL_2(\ZZ_2)$. Let $\psi_{2}:\GL_2(\ZZ_2)\to\mu_2$ be the corresponding homomorphism obtained from factorization of $\psi$ via $\GL_2(\ZZ_2)$. For primes $p\nmid 2D$, let $\chi_p$ be constantly equal to $1$. For odd primes $p\mid D$, let $\chi_p=\chi_{D,p}$ be the lift of the Legendre symbol mod $p$ to $\ZZ_p^{\times}$,  
%for odd primes $p\mid D$, 
i.e.,
\begin{equation*}
    \chi_p:\GL_2(\ZZ_p^{\times})\stackrel{\det}{\longrightarrow}\ZZ_p^{\times}\stackrel{}{\longrightarrow}\mu_2
\end{equation*}
where the last map is the composition of {projection map to $\mathbb{Z}/p\mathbb{Z}$} and the Legendre symbol mod $p$. For prime $2$, let $\chi_2=\chi_{D,2}\cdot\psi_{2}$, where $\chi_{D,2}$, similarly to the Kummer case, is the lift of one of the Dirichlet characters mod $8$ to $\ZZ^{\times}_2$ (if $D$ is odd, then $\chi_{D, 2}$ is trivial). %For odd primes $p\nmid D$, let $\chi_p=1$. 
Therefore, by the above construction of $\chi$, we have the decomposition $\chi=\prod_p\chi_p$.

Let $A_p=\GL_2(\ZZ_p)$ and $A(p^k)=\GL_2(\ZZ/p^k\ZZ)$. The following is an analogous of Proposition \ref{character} for Serre curves.
\begin{proposition}
\label{character-serre}
For a Serre curve $E$, assume the above notations.
Let $\ell(p)$ be the smallest integer $k$ for which $\chi_p$ factors via $A(p^k)$. Then
$$\ell(p)=\left\{\begin{array}{ll}
0&\text{if~} p~ \text{is~ odd~and~}p\nmid D,\\
1&\text{if~} p~ \text{is~ odd~and~}p\mid D,\\
1&\text{if}~ p=2~ \text{and}~ D ~\text{is odd},\\
2&\text{if}~ p=2~ \text{and}~ 4\Vert D,\\
3&\text{if}~ p=2~ \text{and}~ 8\Vert D.\\ 
\end{array}
\right.
$$

\end{proposition}

\begin{proof}
If $p\nmid 2D$, then $\chi_p$ is constanly equal to $1$. Hence, it factors via $A(1)$. If $p$ is odd and  $p\mid D$, then $\chi_p$ is the Legendre symbol mod $p$, and so it factors via $A(p)$, and since it is non-trivial, it does not factor via $A(1)$. The result for $p=2$ follows from the construction of $\chi_2$ described above, noting that the smallest integer $k$ for which    
$\psi_2$ factors via $A(p^k)$ is $k=1$, for $4\Vert D$ the smallest such $k$ is $k=2$, and for $8\Vert D$  the smallest such $k$ is $k=3$.
\end{proof}

We are now ready to prove our last remaining assertion.

\begin{proof}[Proof of Proposition \ref{prop-Serre}]
%Similar to the proof of Theorem \ref{product-kummer-family}, if $\int_A \Tilde{g}\neq0$, b
Following steps similar to the proof of Theorem \ref{product-kummer-family} and by employing  Corollary \ref{cor-after-main1} with $m=2$,  Proposition \ref{character-serre}, and, for integer $n\geq 1$,
\begin{equation*}
\label{|GL|}
    \#\GL_2(\ZZ/n\ZZ)=\prod_{p^e\:\Vert\:n}p^{4e-3}(p^2-1)(p-1)
\end{equation*}
(see \cite[p. 231]{K}) we have the stated product expression.
\end{proof}

\medskip\par
\noindent{\bf Acknowledgements.} 
The authors thank the reviewers for their valuable comments and suggestions.
The authors thank David Basil and Solaleh Bolvardizadeh for help computing the explicit constants $ c_a$ of Proposition\ref{TDPK-formula} for certain values of $a$.

\section*{Declarations}

\subsection*{Funding} This research is partially supported by NSERC.
\subsection*{Data availability statement} Data sharing does not apply to this article as no datasets were generated or analyzed during the current study.

\subsection*{Conflict of interest} The authors have no conflicts of interest to declare relevant to this article's content.

\begin{rezabib} 
\bib{AG}{article}{
   author={Akbary, Amir},
   author={Ghioca, Dragos},
   title={A geometric variant of Titchmarsh divisor problem},
   journal={Int. J. Number Theory},
   volume={8},
   date={2012},
   number={1},
   pages={53--69},
   issn={1793-0421},
   review={\MR{2887882}},
   doi={10.1142/S1793042112500030},
}

\bib{AF}{article}{
   author={Akbary, Amir},
   author={Felix, Adam Tyler},
   title={On the average value of a function of the residual index},
   journal={Springer Proc. Math. Stat.},
   volume={251},
 %  conference={title={Geometry, algebra, number theory, and their information technology applications},},
   %book={series={Springer Proc. Math. Stat.}, volume={251 publisher={Springer, Cham},},
   date={2018},
   pages={19--37},
   review={\MR{3880381}},
%   doi={10.1007/978-3-319-97379-1_2},
}

\bib{artin}{book}{
   author={Artin, Emil},
   title={The collected papers of Emil Artin},
   note={Edited by Serge Lang and John T. Tate},
   publisher={Addison-Wesley Publishing Co., Inc., Reading, Mass.-London},
   date={1965},
   pages={xvi+560 pp. (2 plates)},
   review={\MR{0176888}},
}

\bib{BLSW}{article}{
   author={Bach, Eric},
   author={Lukes, Richard},
   author={Shallit, Jeffrey},
   author={Williams, H. C.},
   title={Results and estimates on pseudopowers},
   journal={Math. Comp.},
   volume={65},
   date={1996},
   number={216},
   pages={1737--1747},
   issn={0025-5718},
   review={\MR{1355005}},
}

 \bib{cojocaru:tdp}{article}{
   author={Bell, Renee},
   author={Blakestad, Clifford},
   author={Cojocaru, Alina Carmen},
   author={Cowan, Alexander},
   author={Jones, Nathan},
   author={Matei, Vlad},
   author={Smith, Geoffrey},
   author={Vogt, Isabel},
   title={Constants in Titchmarsh divisor problems for elliptic curves},
   journal={Res. Number Theory},
   volume={6},
   date={2020},
   number={1},
   pages={Paper No. 1, 24},
   issn={2522-0160},
   review={\MR{4041152}},
   doi={10.1007/s40993-019-0175-9},
}
 
 \bib{cox}{book}{
   author={Cox, David A.},
   title={Primes of the form $x^2 + ny^2$},
   series={Pure and Applied Mathematics (Hoboken)},
   edition={2},
   note={Fermat, class field theory, and complex multiplication},
   publisher={John Wiley \& Sons, Inc., Hoboken, NJ},
   date={2013},
   pages={xviii+356},
   isbn={978-1-118-39018-4},
   review={\MR{3236783}},
   doi={10.1002/9781118400722},
}
	 
 \bib{davenport}{book}{
   author={Davenport, Harold},
   title={Multiplicative number theory},
   series={Graduate Texts in Mathematics},
   volume={74},
   edition={3},
   note={Revised and with a preface by Hugh L. Montgomery},
   publisher={Springer-Verlag, New York},
   date={2000},
   pages={xiv+177},
   isbn={0-387-95097-4},
   review={\MR{1790423}},
} 

\bib{felix-murty}{article}{
   author={Felix, Adam Tyler},
   author={Murty, M. Ram},
   title={A problem of Fomenko's related to Artin's conjecture},
   journal={Int. J. Number Theory},
   volume={8},
   date={2012},
   number={7},
   pages={1687--1723},
   issn={1793-0421},
   review={\MR{2968946}},
   doi={10.1142/S1793042112500984},
}

\bib{FA}{book}{
   author={Fried, Michael D.},
   author={Jarden, Moshe},
   title={Field arithmetic},
   series={Ergebnisse der Mathematik und ihrer Grenzgebiete. 3. Folge. A
   Series of Modern Surveys in Mathematics [Results in Mathematics and
   Related Areas. 3rd Series. A Series of Modern Surveys in Mathematics]},
   volume={11},
   edition={3},
   note={Revised by Jarden},
   publisher={Springer-Verlag, Berlin},
   date={2008},
   pages={xxiv+792},
   isbn={978-3-540-77269-9},
   review={\MR{2445111}},
}

\bib{Hooley}{article}{
   author={Hooley, Christopher},
   title={On Artin's conjecture},
   journal={J. Reine Angew. Math.},
   volume={225},
   date={1967},
   pages={209--220},
   issn={0075-4102},
   review={\MR{207630}},
   doi={10.1515/crll.1967.225.209},
}

\bib{K}{book}{
   author={Koblitz, Neal},
   title={Introduction to elliptic curves and modular forms},
   series={Graduate Texts in Mathematics},
   volume={97},
   edition={2},
   publisher={Springer-Verlag, New York},
   date={1993},
   pages={x+248},
   isbn={0-387-97966-2},
   review={\MR{1216136}},
   doi={10.1007/978-1-4612-0909-6},
}

\bib{kowalski}{article}{
   author={Kowalski, E.},
   title={Analytic problems for elliptic curves},
   journal={J. Ramanujan Math. Soc.},
   volume={21},
   date={2006},
   number={1},
   pages={19--114},
   issn={0970-1249},
   review={\MR{2226355}},
}

%\bib{lang}{book}{
%   author={Lang, Serge},
%   title={Algebra},
%   series={Graduate Texts in Mathematics},
%   volume={211},
%   edition={3},
%   publisher={Springer-Verlag, New York},
%   date={2002},
%   pages={xvi+914},
%   isbn={0-387-95385-X},
%   review={\MR{1878556}},
%   doi={10.1007/978-1-4613-0041-0},
%} 

\bib{L}{article}{
   author={Laxton, R. R.},
   title={On groups of linear recurrences. I},
   journal={Duke Math. J.},
   volume={36},
   date={1969},
   pages={721--736},
   issn={0012-7094},
   review={\MR{0258781}},
}

\bib{lenstra-stevenhagen-moree}{article}{
   author={Lenstra, H. W., Jr.},
   author={Moree, P.},
   author={Stevenhagen, P.},
   title={Character sums for primitive root densities},
   journal={Math. Proc. Cambridge Philos. Soc.},
   volume={157},
   date={2014},
   number={3},
   pages={489--511},
   issn={0305-0041},
   review={\MR{3286520}},
   doi={10.1017/S0305004114000450},
}

\bib{MS}{article}{
   author={Moree, P.},
   author={Stevenhagen, P.},
   title={Computing higher rank primitive root densities},
   journal={Acta Arith.},
   volume={163},
   date={2014},
   number={1},
   pages={15--32},
   issn={0065-1036},
   review={\MR{3194054}},
   doi={10.4064/aa163-1-2},
}

\bib{Papa}{article}{
   author={Pappalardi, F.},
   title={On Hooley's theorem with weights},
   note={Number theory, II (Rome, 1995)},
   journal={Rend. Sem. Mat. Univ. Politec. Torino},
   volume={53},
   date={1995},
   number={4},
   pages={375--388},
   issn={0373-1243},
   review={\MR{1452393}},
}

\bib{rudin}{book}{
   author={Rudin, Walter},
   title={Real and complex analysis},
   edition={3},
   publisher={McGraw-Hill Book Co., New York},
   date={1987},
   pages={xiv+416},
   isbn={0-07-054234-1},
   review={\MR{924157}},
}

\bib{course-in-arithmetic}{book}{
   author={Serre, J.-P.},
   title={A course in arithmetic},
   series={Graduate Texts in Mathematics, No. 7},
   note={Translated from the French},
   publisher={Springer-Verlag, New York-Heidelberg},
   date={1973},
   pages={viii+115},
   review={\MR{0344216}},
}

\bib{serre}{article}{
   author={Serre, Jean-Pierre},
   title={Propri\'{e}t\'{e}s galoisiennes des points d'ordre fini des courbes
   elliptiques},
   language={French},
   journal={Invent. Math.},
   volume={15},
   date={1972},
   number={4},
   pages={259--331},
   issn={0020-9910},
   review={\MR{387283}},
   doi={10.1007/BF01405086},
}

\bib{stephens}{article}{
   author={Stephens, P. J.},
   title={Prime divisors of second-order linear recurrences. I},
   journal={J. Number Theory},
   volume={8},
   date={1976},
   number={3},
   pages={313--332},
   issn={0022-314X},
   review={\MR{0417081}},
   doi={10.1016/0022-314X(76)90010-X},
}

\bib{S}{article}{
   author={Stevenhagen, Peter},
   title={The correction factor in Artin's primitive root conjecture},
   language={English, with English and French summaries},
   note={Les XXII\`emes Journ\'{e}es Arithmetiques (Lille, 2001)},
   journal={J. Th\'{e}or. Nombres Bordeaux},
   volume={15},
   date={2003},
   number={1},
   pages={383--391},
}

\bib{wagstaff}{article}{
   author={Wagstaff, Samuel S., Jr.},
   title={Pseudoprimes and a generalization of Artin's conjecture},
   journal={Acta Arith.},
   volume={41},
   date={1982},
   number={2},
   pages={141--150},
   issn={0065-1036},
   review={\MR{674829}},
   doi={10.4064/aa-41-2-141-150},
}

\end{rezabib}

\end{document}